\documentclass[11pt,a4paper,reqno]{amsart}

\usepackage{amsfonts,amsthm,amscd,epsfig,amsmath,amssymb,enumerate ,psfrag }

\numberwithin{equation}{section}

\DeclareMathSymbol{\leqslant}{\mathalpha}{AMSa}{"36} % nicer `smaller or equal'
\DeclareMathSymbol{\geqslant}{\mathalpha}{AMSa}{"3E} % nicer `larger or equal'
\DeclareMathSymbol{\eset}{\mathalpha}{AMSb}{"3F}     % nicer `emptyset'
\renewcommand{\leq}{\;\leqslant\;}                   % redef. of < or =
\renewcommand{\geq}{\;\geqslant\;}                   % redef. of > or =
\newcommand{\be}{\begin{equation}}

\def\1{\ifmmode {1\hskip -3pt \rm{I}} \else {\hbox {$1\hskip -3pt \rm{I}$}}\fi}

\newtheorem{Th}{Theorem}[section]
\newtheorem{Le}[Th]{Lemma}
\newtheorem{Pro}[Th]{Proposition}

\newtheorem{Def}[Th]{Definition}
\newtheorem{Rem}{Remark}
\newcommand{\cM}{\ensuremath{\mathcal M}}

\newcommand{\bbL}{{\ensuremath{\mathbb L}} }

\newcommand{\bbR}{{\ensuremath{\mathbb R}} }

\newcommand{\bbT}{{\ensuremath{\mathbb T}} }

\newcommand{\bbZ}{{\ensuremath{\mathbb Z}} }

    \let\d=\delta  \let\e=\varepsilon
 \let\g=\gamma       \let\l=\lambda
          \let\p=\pi  
  \let\s=\sigma \let\t=\tau   
  
\let\D=\Delta   \let\G=\Gamma   
\let\O=\Omega

\title[A representation formula for large deviations on the 1D torus]
{A representation formula for large deviations rate functionals of
invariant measures on the one dimensional torus. }

\author{A. Faggionato}
\address{Alessandra Faggionato. Dipartimento di Matematica ``G. Castelnuovo", Universit\`a ``La
  Sapienza''. P.le Aldo Moro  2, 00185  Roma, Italy. e--mail:
  faggiona@mat.uniroma1.it}

\author{D. Gabrielli}
 \address{Davide Gabrielli.  Dipartimento di Matematica, Universit\`a
 dell'Aquila,   67100 Coppito, L'Aquila, Italy. e--mail:
 gabriell@univaq.it}

\thanks{Work supported by the grant
 PRIN 20078XYHYS$\underline{\
}$003 and by
 the European Research Council through the ``Advanced
Grant'' PTRELSS 228032}

\begin{document}

\begin{abstract}

We consider a generic diffusion on the 1D torus and give a simple
representation formula for the large deviation rate functional of
its invariant probability measure, in the limit of vanishing noise.
Previously, this rate functional had been characterized  by M.I.\
Freidlin and A.D.\ Wentzell as solution of a rather complex
optimization problem. We discuss this last problem in full
generality and show that it leads to our formula.
We express  the rate
functional by means  of a geometric transformation that,
with a Maxwell-like construction, creates flat regions.
%This reveals the presence of a phase
%$transition.
% We discuss in detail the structure of this geometric
%transformation.

 We then consider  piecewise deterministic Markov
processes on the 1D torus and show that the corresponding large
deviation rate functional for the stationary distribution is
obtained by applying the same transformation. Inspired by this, we
prove a universality result showing that the transformation
generates viscosity solution of stationary Hamilton--Jacobi equation
associated to any Hamiltonian $H$ satisfying suitable weak
conditions.
 % In addition, we discuss some general features of
% the static Large Deviation rate functional
% both  for   PDMPs  on $\bbT$ in the limit  of diverging  jump frequency and
%  for diffusions  on $\bbT$  in the limit  of vanishing  noise.
% Finally, we discuss possible phase transitions. ....

\medskip

\noindent {\sl Key words}: diffusion, piecewise deterministic Markov
process, invariant measure, large deviations, Hamilton--Jacobi
equation.

\medskip

\noindent {\sl AMS 2000 subject classification}:
82C05 % Classical dynamic and nonequilibrium statistical mechanics
60J60 % diffusion
60F10 % large deviations
\end{abstract}

\maketitle

\section{Introduction}

We consider two different  random dynamical systems on  the one
dimensional torus $\bbT$ that, in suitable regimes, can be thought
of as random perturbations of deterministic dynamical systems. The
first one is a diffusion on $\bbT$, with small noise of intensity
$\e$. This system has an invariant distribution $\mu^\e$, whose
large deviation (LD) functional has been expressed by
  M.I.\
Freidlin and A.D.\ Wentzell     as solution of an optimization
problem \cite{FW}. As discussed in \cite{FW}, already for a very
 simple example of diffusion on $\bbT$ with   velocity field having only three
attractor points, the solution of this optimization problem requires
a rather long procedure.

The second system we consider  is given by a piecewise deterministic
Markov process (PDMP) on $\bbT$: the  state  is described by a
pair $(x, \s)\in \bbT\times \{0,1\} $, the continuous variable $x$
follows a piecewise deterministic dynamics with nonvanishing
$\s$--dependent  velocity field, the discrete variable $\s$ evolves
by an $x$--dependent stochastic jump dynamics and the two resulting
evolutions are fully--coupled. When the jump rates of the discrete
variable are multiplied by a factor $\l$, in the limit $\l\to
+\infty$ the evolution of the continuous variable $x$ is well
approximated by a deterministic dynamical system (cf. \cite{FGR1},
\cite{K}). In \cite{FGR2},  an expression of the probability
distribution  $ \mu^\l$ of the continuous variable $x$ in the steady
state is computed up to a normalization constant. In addition, from
this expression    the LD  functional of $ \mu^\l$ is computed in
the limit of diverging frequency jumps (i.e. $\l \rightarrow
\infty$). The resulting formula is simple and concise.

Although the two models are not similar, we show here that the LD
rate functionals for the measures $\mu^\e$ and $\mu^\l$ share a
common structure. In particular they admit a very simple expression,
that we further investigate. We then show that the optimization
problem of M.I.\ Freidlin and A.D.\ Wentzell  leads indeed to the
same expression, by solving this optimization problem  in the
general case.

For both models the rate functional is given by a geometric
transformation applied to a specific non periodic function. The
result is a periodic function, whose graph differs from the original one
due to new flat regions. The function to be
transformed is model dependent while the transformation is always
the same. In this sense our result is universal. We discuss this
issue in terms of Hamilton--Jacobi equations. More precisely we
discuss the regularity properties of the functions obtained by this
procedure showing that they are viscosity solution of a suitable
class of Hamilton--Jacobi equations.

When the models are reversible the transformation reduces to the
identity. In all the other cases intervals on which the rate
functional is constant appear. This reveals the presence of a phase
transition. This kind of stationary non equilibrium states have a
physical relevance and have been created and studied experimentally
(see for example\cite{GPCCG}).

%In our investigation the assumption of dimension one is essential.
%We expect that in higher dimension results can be obtained in view
%of recent developments of the geometric theory of Hamilton--Jacobi
%equation \cite{FS}.

\section{Models and results}

Without loss of generality, we  think of $\bbT$ as the interval
$[0,1]$ with identification of the extreme points $0$ and $1$.

\subsection{Models}
The first model we consider is a generic diffusion  $\bigl(
X_t^\e\bigr) _{ t\geq 0}$   described by the equation
\begin{equation}\label{agostino}
\dot{X}_t^\e =b( X_t^\e)+ \e \dot{w}_t\,,
\end{equation}
where $b: \bbT\rightarrow \bbR$ is a Lipschitz  continuous  vector
field, $w_t$ is a Wiener process  and $\e$ is a positive parameter.
A detailed analysis of the above diffusion, as well as of diffusions
on generic manifolds $\cM$, in the limit $\e \downarrow 0$ is given
in \cite{FW}.  In the case $\cM=\bbT$, Theorem 4.3 in Section 6.4 of
\cite{FW} under the assumption that the closed set $\{x\in \bbT\,:\,
b(x)=0\}$ has a finite number of connected components gives
\begin{equation}\label{nonnabruna}
\lim _{\e \downarrow 0 } -\e^2 \log \mu^\e (x) =W(x)- \min _{y \in
\bbT} W(y) \,, \qquad x \in \bbT\,,
\end{equation}
where $\mu^\e(x) dx $ denotes the invariant probability measure of
the diffusion \eqref{agostino} and where the continuous function $W$
is described in \cite{FW}  by  a rather complex variational
characterization
 that
we  recall in Section \ref{fufi}.
 The r.h.s. of \eqref{nonnabruna} is the LD  rate functional for $\mu^\e$.
Here and in all the paper we state our large deviations results in
the simple and direct formulation used in \eqref{nonnabruna}. Of
course we mean that $\mu^\e(x)$ is a continuous version of the
density of the invariant measure.

A very simple expression both of the invariant measure and of the LD
rate functional can be given. To this aim, in the formulas below  we
will  think  of the  field $b(\cdot)$ also as a periodic function on
$\bbR$, with periodicity $1$. With this convention and without
requiring that the set $\{b=0\}$ has a finite number of connected
components, we get:
 \begin{Pro}\label{pierpi}
Define the function $S:\bbR \rightarrow \bbR$ as $ S(x)= -2\int _0
^x b(s) ds$.
 Then,
\begin{equation}
\mu^\epsilon(x)=\frac{1}{c(\epsilon)}
\int_x^{x+1}e^{\epsilon^{-2}\left(S(y)-S(x)\right)}\,dy\,,\label{keyformula}
\end{equation}
where $c(\e)$ is the normalization constant \begin{equation}
c(\epsilon)=
 \int_0^1dx \int_x^{x+1}
e^{\epsilon^{-2}\left(S(y)-S(x)\right)}dy \,.\label{keyformula_zac}
\end{equation}
In particular, it holds
\begin{eqnarray}
\lim _{\e \downarrow 0 } -\e^2\log \mu^\e (x) &=&\max_{x'\in [0,1]}
\max _{y'\in [x',x'+1]} \bigl(S(y')-S(x') \bigr) -  \max_{y \in
[x,x+1]}\bigl( S(y)-S(x) \bigr)\nonumber \\
&=&\min_{y \in [x,x+1]}\bigl( S(x)-S(y) \bigr)-\min_{x'\in [0,1]}
\min _{y'\in [x',x'+1]} \bigl(S(x')-S(y')
\bigr)\,.\label{nonnabrunabis}
\end{eqnarray}
\end{Pro}
Note that  $S(a+1)-S(a)=- \int _0^1 2 b(s) ds$, and in particular
the difference  does not depend on $a\in \bbR$. Hence, given $x \in
\bbT$,  the expression $\max _{y \in [  x,   x+1]} \bigl(S(y)-S(
x)\bigr)$ does not
  generate any confusion, both if we think $x
\in \bbT \hookrightarrow \bbR$ by identifying $\bbT$ with $[0,1)$,
and if we think $\max _{y \in [  x,   x+1]}\bigl( S(y)-S(  x)\bigr)$
as $\max _{y \in [ \bar x, \bar x+1]} \bigl(S(y)-S( \bar  x)\bigr) $
with $\bar x \in \bbR$ such that $ \p (\bar x )=x$, where $\p:
\bbR\rightarrow \bbT$ denotes the canonical projection.  In what
follows, when writing $\max _{y \in [ x, x+1]} \bigl(S(y)-S(
x)\bigr)$ with $x \in \bbT$ we will mean any of the above
interpretations. The same considerations hold if we consider the
minimum instead of the maximum.

\begin{proof}[Proof of Proposition  \ref{pierpi}]
The proof is elementary. The fact that \eqref{keyformula} is the
density of the invariant measure follows by a direct computation.
See for example \cite{MNW} where a similar expression has been
obtained. Then \eqref{nonnabrunabis} follows from \eqref{keyformula}
and \eqref{keyformula_zac} by a direct application of the Laplace
Theorem \cite{FW}.
\end{proof}

 As a consequence of Proposition \ref{pierpi}, the r.h.s. of
\eqref{nonnabrunabis} coincides with the r.h.s. of
\eqref{nonnabruna}, where we recall that the function $W$ is
characterized as the solution of the Freidlin-Wentzell variational
problem. This fact is not evident. The general solution of the
variational problem determining $W$ on the 1D torus is described in
detail in Theorem \ref{ruspabis}, which is stated only  in Section
\ref{otto} after introducing some preliminaries. Its proof is given
in Section \ref{fatica} and  is independent from Proposition
\ref{pierpi}. The identification of $W$ with the r.h.s. of
\eqref{nonnabrunabis} is stated in Theorem \ref{ruspa}.

\medskip

The second model we discuss is a PDMP on the 1D torus $\mathbb T$.
Let $F_0,F_1: \bbT \rightarrow \bbR$ be Lipschitz continuous
fields. In addition, let $r(0,1|\cdot), r(1,0|\cdot)$ be positive
continuous functions on $\bbT$. Given the parameter $\l>0$, we
denote by $\{(X^\l_t, \s^\l_t)\,:\, t\geq 0\}$ the stochastic
process with states in $\bbT\times \{0,1\}$ whose generator is given
by
\begin{equation} \bbL_\l f (x,\s)= F_\sigma(x)\cdot \nabla
f(x,\sigma)+\l
r(\sigma,1-\sigma|x)\left(f(x,1-\sigma)-f(x,\sigma)\right)\,,
\end{equation}
for all $ (x, \s) \in \bbT \times \{0,1\}$. The above process is a
generic PDMP on the torus $\bbT$ (cf. \cite{D} for a detailed
discussion on PDMPs). Following \cite{FGR1}, \cite{FGR2}, we call
$x$ and $\s$  the mechanical and the chemical state of the system,
respectively. The dynamics can be roughly described as follows.
Given the initial state $(x_0,\sigma_0)\in \O\times \G$, consider
the positive  random variable $\tau_1$ with distribution
\begin{equation*}
\mathbb P(\tau_1> t)=
e^{-\lambda\int_0^t r(\sigma_0,1-\sigma_0|x_0(s))ds}\,, \qquad  t\geq 0\,,
\end{equation*}
where  $x_0(s)$ is the solution of the Cauchy problem
\begin{equation}
\left\{
\begin{array}{l}
\dot{x}= F_{\sigma_0}(x)\,, \\
x(0)=x_0 \,.
\end{array}
\right.\label{Cauchy}
\end{equation}
The evolution of the system  in the time interval $[0,\tau_1)$ is
given by $(x_0(s),\sigma_0)$. At time $\tau_1$ the chemical state
changes, i.e. $\sigma^\l _{\tau_1}= 1- \sigma _0$, and the dynamics
starts afresh from the state $(x(\tau_1), \sigma^\l _{\tau_1}  )$.
Note that the mechanical trajectory is continuous and piecewise
$C^1$. Moreover, if the jump rates $r(0,1|x)$ and $r(1,0|x) $ do not
depend on $x$, then the process $(\s_t^\l \,:\, t \geq 0)$ reduces
to a continuous--time Markov chain with jump rates $r(0,1), r(1,0)$,
independent from $(X^\l_t\,:\, t \geq 0 ) $. In general, the
chemical and the mechanical evolutions are fully--coupled.

In \cite{FGR1} (see also \cite{FGR2}) an averaging principle has
been proved. In the limit of high frequency of the chemical jumps (i.e. $\l\to \infty$),
the mechanical variable $x$ behaves deterministically according to an ODE with a
suitable averaged vector field $\bar F$. In this sense a PDMP can be
thought of as a stochastic perturbation of the deterministic
dynamical system $\dot x=\bar F(x)$.

For simplicity, we restrict to the case of non vanishing force
fields $F_0,F_1$. Then for each $\l>0$, there exists a unique
invariant measure $\tilde \mu^\l $ for the PDMP $( X^\l_t , \s
^\l_t)$ and it has the form   $\tilde \mu^\l(x,\sigma)= \tilde
\mu^\l_0 (x) dx \d_{\s,0}+\tilde \mu^\l _1 (x) dx \d _{\s,1}$, where
$\d$ is the Kronecker delta (cf. Theorem (34.19), p.118 and Theorem
3.10, p.130 in \cite{D}[Section 34] together with \cite{FGR2}). Let
us observe now the evolution of the mechanical state alone in the
steady state. We set
$$   \mu^\l (x)=\tilde  \mu^\l _0 (x)+\tilde \mu^\l _1(x) $$
for the probability density at $x$ of the mechanical variable in the
steady state. The following result holds \cite{FGR2}:
\begin{Pro}\label{pepo} Suppose that $F_0$ and $F_1$ are Lipschitz continuous
non vanishing   fields and define the function $S:\bbR\to \mathbb R$
as
\begin{equation}S(x)= \int _0^ x \left(\frac{r(0,1|y)}{F_0(y)} +
\frac{ r(1,0|y)}{F_1(y)} \right) dy\,.
%=\int _0 ^x \frac{\bar F(z)}{F_0(z)F_1(z)[r(0,1|z)+r(1,0|z)]} dz \,.
\end{equation}
Then \begin{equation}   \mu ^\l (x)=\frac{1}{Z^\l} \int_x^{x+1}
\Big[ \frac{r(1,0|y)}{ F_0(x)F_1(y) }+ \frac{ r(0,1|y) }{
F_1(x)F_0(y) } \Big] e^{\l \bigl( S(y)-S(x) \bigr) } dy\,,
\end{equation}
where $Z^\l$ denotes the normalization constant.  In particular,  it
holds
\begin{equation}
\begin{split}
\lim _{\l \uparrow \infty}- \frac{1}{\l}\ln \mu^\l (x)&=
 \max_{x'\in [0,1]} \max _{y'\in [x',x'+1]} \bigl(S(y')-S(x') \bigr)
-  \max_{y \in [x,x+1]}\bigl( S(y)-S(x) \bigr) \\
&=\min_{y \in [x,x+1]}\bigl( S(x)-S(y) \bigr)-\min_{x'\in [0,1]}
\min _{y'\in [x',x'+1]} \bigl(S(x')-S(y') \bigr)\,. \text{$ $ }\label{nonnobile}
\end{split}
\end{equation}
\end{Pro}
The analogy with Proposition \ref{pierpi} is evident. We refer the
reader to \cite{FGR2} for a proof of the above results.

\subsection{Common geometric structure of the LD functionals.}

We can finally describe the common structure behind the LD
functionals of $\mu^\e$ and $\mu^\l$:

\begin{Th}\label{ruspa} Given a continuous  function  $F: \bbT
\rightarrow \bbR$, define for all $x \in \bbR $
\begin{align}
& S(x) = \int _0 ^x F(s) ds  \,,\\
&  \Phi (x)= -\max_{y \in [x,x+1]} \bigl( S(y)-S(x) \bigr)= \min _{y
\in [x,x+1]} \bigl( S(x) - S(y) \bigr)\,.
\end{align}
Then the following holds:

\begin{itemize}

\item [ (i)] $\Phi$ is Lipschitz continuous  and periodic with unit period.
If $S$ is monotone, then $\Phi $ is constant and equals
$\min\left\{0,-S(1)\right\}$.
 If
 $S(1) =0$, then $S$ is periodic and $\Phi= S$ up to an additive constant. If
 $ S(1) \not = 0$ and $S$ is not monotone, then the set
 \begin{equation}\label{urca}
U= \{ x \in \bbR\,:\, \Phi(x)\not = \min \{0,-S(1)\} \}
 \end{equation}
is an open subset $U
 \subset \bbR$ such that  $\bbR\setminus U$ has nonempty interior
 part.  On each connected component of $U$ it holds   $\Phi=S$ up to an
 additive constant,
 i.e.
\begin{equation}\label{gattone} \Phi (x) = S(x) - S(a)+ \Phi (a)\,, \qquad \forall x \in (a,b)
 \subset U\,.
\end{equation}
 On $\bbR\setminus U$ the function  $\Phi$ is constant and satisfies
 \begin{equation}\label{patroclo}
 \Phi(x)= \min \{0,- S(1) \}\,,
 \qquad \forall x \in \bbR \setminus U\,.
 \end{equation}
 Moreover, $\Phi$ reaches its maximum on $\bbR\setminus U$:
 \begin{equation}\label{ciampi}
\max_{x\in [0,\infty) }\Phi(x)=\max _{x\in [0,1]}\Phi(x)=\min \{0,
-S(1) \}\,.
 \end{equation}

\item[
(ii)]  Suppose that the set $\{ x \in \bbT\,:\, F(x)=0\}$ has a
finite number of connected components. Then,
\begin{equation}\label{ziofabri}
\Phi (x) - \min_{x' \in [0,1]} \Phi (x')  = W(x)- \min_{y \in
\mathbb T} W(y)
\end{equation}
where $W$ is the function entering in \eqref{nonnabruna} and defined
by an optimization problem in \cite{FW}, taking $b(x)=-F(x)/2$.
%Call $K_1, K_2, \dots,
%K_{ell}$ the
% connecting components of the set $\{ F=0\}$ given by local minimum points of $S$  and
% call $A_i$ the connected component in $\{F=0\}$ given by  the local
% maximum points of $S$ between $K_i$ and $K_{i+1}$. Then it holds
% \begin{equation}
\end{itemize}
\end{Th}
The  map from $S$ to $\Phi$ is the geometric transformation
mentioned in the Introduction.

\begin{Rem}
It is interesting to observe that if we define
\begin{equation}
\tilde{\Phi}(x):=\inf_{y\in[x,x+1]}\left(\int_x^{y}F_-(s)ds+\int_{y}^{x+1}
F_+(s)ds
 \right)\,, \label{gelmini}
\end{equation}
then $\tilde \Phi$ and $\Phi$ differ by a constant. In fact, we have
\begin{eqnarray}
& &\Phi(x)=\min_{y\in[x,x+1]}(S(x)-S(y))=\min_{y\in[x,x+1]}\left(\int_x^yF_-(s)ds-\int_x^yF_+(s)ds\right)\nonumber \\
& &=\min_{y\in[x,x+1]}\left(\int_x^yF_-(s)ds+
\int_y^{x+1}F_+(s)ds\right)-\int_0^1F_+(s)ds\nonumber \\
& &=\tilde{\Phi}(x)-\int_0^1F_+(s)ds\,.\nonumber
\end{eqnarray}
Indeed, comparing with \eqref{pisolo} in Proposition
\ref{biancaneve}, one gets $\tilde{\Phi}= W$ under the assumptions
of Theorem \ref{ruspa} (ii).
\end{Rem}

Theorem \ref{ruspa} suggests  a simple algorithm  to determine the
graph of the function $\Phi$. To avoid trivial cases, we assume that
$S(1)\not= 0$ and that  $S$ is not monotone. For simplicity of
notation we suppose that the set $\{F=0\}$ has finite cardinality.

\noindent $\bullet$ {\em Case   $S(1)>0$}.
 We can always find a point of
local maximum $b$ such that $S(b+1)= \max _{y \in [b,b+1]} S(y)$.
This point $b$ can be found as follows: let $a\in \bbR $ be any
point of local maximum  for $S$, then let  $b \in [a,a+1]$ be such
that $S(b) = \max _{y \in [a,a+1]} S(y)$. This trivially implies
that
$$ S(b+1)=S(1)+S(b) \geq
\begin{cases}
 S(1)+ \max _{y \in [b,a+1]} S(y)\,,\\
 S(1)+ \max_{y \in [a,b]} S(y)= \max _{y \in [a+1,b+1]} S(y)\,.
\end{cases}
$$
The above inequalities imply that $S(b+1) = \max _{y \in [b,b+1]}
S(y)$. If $b\in (a,a+1)$ then trivially $b$ is a point of local
maximum for $S$. Otherwise, it must be $b=a+1$ (since
$S(a+1)=S(a)+S(1)>S(a)$) and in particular $b$ is again a point of
local maximum for $S$ since $a$ and therefore $a+1$ satisfy this
property.

 The following algorithm shows how to construct the function
$\Phi$ on the interval $[b,b+1]$. Due to the periodicity of $\Phi$
this construction extends to all $\mathbb R$.

 Let $\G=\{x_i: 1\leq i \leq m  \}$ be the
set of points of local maximum for $S$  in $[b,b+1]$ (note that
$b,b+1 \in \G$). Let $ L_1 = S( b+1)$, $ z_1= \min \{ x \in \G \,:\,
S(x)= L_1\}$, and define inductively
$$L_k = \max \{ S(x): x \in \G ,
x<z_{k-1} \}\,, \qquad z_k =\min \{ x \in \G\,:\, S(x)= L_k\} $$ for
all $k\geq 2 $ such that the set  $\{ S(x): x \in \G , x<z_{k-1} \}$
is nonempty. At the end, we get $n$ levels $L_1, \dots, L_n$ and
points $z_n<z_{n-1} < \cdots < z_1$, which are all points of local
maximum for $S$ on $[b,b+1]$. In addition, it must be $z_n=b$. For
each $i:1\leq i <n$ let $y_i$ be the point
$$ y_i= \max \{x \in [b, z_i]\,:\, S(x) = L_{i+1}\}
\,.
$$
Then the continuous function $\Phi$ on $[b,b+1]$ is obtained  by the
following rules. Fix the value $\Phi(b)=-S(1)$. Set $V= \cup
_{i=1}^{n-1} [y_i,z_i]$, then the function $\Phi$ must be equal to
the constant $-S(1)$ on $V$ and must satisfy $\nabla\Phi= \nabla S$
on $[b,b+1] \setminus V$. See figure \ref{mayo1000}.

We do not give a formal proof of the above algorithm, it can be
easily obtained for example  from the following alternative
representation of the function $\Phi$ when $x$ belongs to the
special unitary period $[b,b+1]$
\begin{eqnarray}
& &\Phi(x)=\min_{y\in[x,x+1]}(S(x)-S(y))=S(x)-\max_{y\in[x,x+1]}S(y)\nonumber\\
&
&=S(x)-\max_{y\in[b+1,x+1]}S(y)=S(x)-S(1)-\max_{y\in[b,x]}S(y)\,.\label{fb}
\end{eqnarray}
Note that in the last term we are maximizing only over the interval
$[b,x]$.
%Since $[b,x_1]\subset [b,x_2]$ when $x_1<x_2$ then
%$\max_{y\in[b,x]}S(y)$ is a monotone increasing function whose
%gradient either is zero or coincide with the gradient of $S$.
The
algorithm follows easily.

The algorithm can be summarized by the following simple and
intuitive procedure. Think of  the graph of $S$ in Figure
\ref{mayo1000} as a mountain profile and imagine also that light is
arriving from the left with rays parallel to the horizontal axis. On
a point $x$ such that the corresponding point on the mountain
profile is in the shadow we have $\nabla \Phi=\nabla S$. On a point
$x$ such that the corresponding point on the mountain profile is
enlightened, we have $\nabla \Phi=0$. Note in particular that flat
intervals are always to the left of some local maxima.

\begin{figure}[!ht]
    \begin{center}
       \psfrag{a}[l][l]{$b$}
      \psfrag{g}[l][l]{$z_1=b+1$}
       \psfrag{m}[l][l]{$L_1$}
 \psfrag{zs}[l][l]{$L_2$}
  \psfrag{v}[l][l]{$L_3$}
 \psfrag{u}[l][l]{$L_4=0$}
 \psfrag{f}[l][l]{$y_1$}
 \psfrag{e}[l][l]{$z_2$} \psfrag{d}[l][l]{$y_2$}
 \psfrag{c}[l][l]{$z_3$} \psfrag{b}[l][l]{$y_3$}
\psfrag{a}[l][l]{$z_4=b$}
\psfrag{R}[l][l]{$-S(1)$}
\psfrag{U}[l][l]{Function $S$}
\psfrag{V}[l][l]{Function $\Phi$}
 \includegraphics[width=11cm]{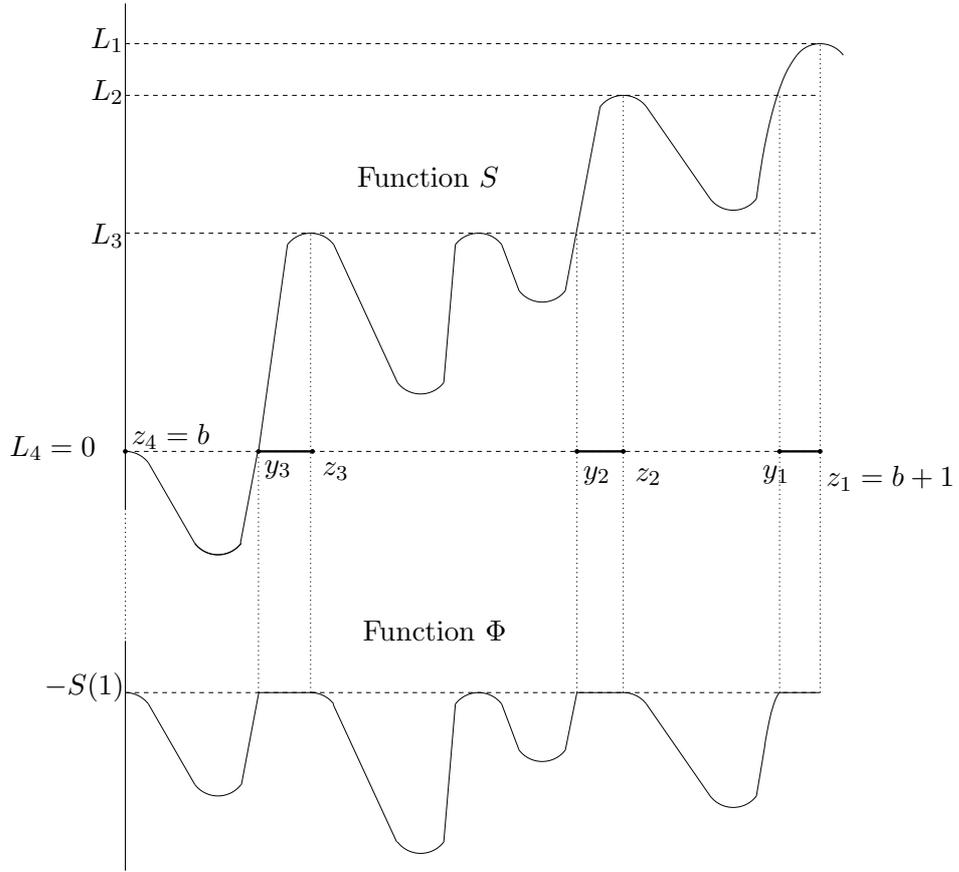}
      \caption{Algorithm to determine the flat regions  and the function $\Phi$: case   $S(1)>0$.}
    \label{mayo1000}
    \end{center}
 \end{figure}

\noindent $\bullet$ {\em Case  $S(1)<0$}.  The algorithm in this
case is similar to the case $S(1)>0$, it is enough to  inverte left
with right. More precisely, first one determines a point of local
maximum $b$ such that $S(b)= \max _{y \in [b,b+1]} S(y)$. One
defines $\G$ as above, let $ L_1 = S( b)$, $ z_1= \max \{ x \in \G
\,:\, S(x)= L_1\}$, and define inductively
$$L_k = \max \{ S(x): x \in \G ,
x>z_{k-1} \}\,, \qquad z_k =\max  \{ x \in \G\,:\, S(x)= L_k\} $$
for all $k\geq 2 $ such that the set  $\{ S(x): x \in \G , x>z_{k-1}
\}$ is nonempty. It must be $z_n=b+1$. For each $i:1\leq i <n$ let
$y_i$ be the point
$$ y_i= \min \{x \in [z_i, b+1]\,:\, S(x) = L_{i+1}\}
\,.
$$
Setting $\Phi(b)=-S(1)$ and   $V=  \cup _{i=1}^{n-1} [z_i,y_i]$, the
conclusion is the same as in the case $S(1)>0$.

Also in this case an interpretation of the algorithm in terms of
mountains and light can be given. The difference is that light is
arriving now from the right. Note also that in this case flats
intervals are always located at the right of some local maxima.

\subsection{Hamilton--Jacobi equations and universality}
Finally, we illustrate the relation of the function
 $\Phi$ with Hamilton-Jacobi equations:

\begin{Th}\label{genteo} Given a continuous function $F:\bbT \to
\bbR$,  let $H(x,p)$, with $(x,p)\in \mathbb T\times \mathbb R$, be
a Hamiltonian such that
\begin{itemize}
\item[(A)] $H(x,\cdot)$ is a convex function for any $x\in \mathbb T$,

\item[(B)]$H(x,0)=H(x,F(x))=0$ for any $x\in \mathbb T$.

\end{itemize}

Then the function $\Phi(x)$  defined in Theorem \ref{ruspa} is a
viscosity solution of the Hamilton-Jacobi equation
\begin{equation}
H(x,\nabla \Phi(x))=0\,, \qquad  x\in \mathbb T\,. \label{HJgen}
\end{equation}
\end{Th}
Hypothesis (A) is   a rather weak assumption since usually the
Hamiltonian is obtained as  Legendre transform of the convex
Lagrangian function. Theorem \ref{genteo} states a universality
property: independently from  the exact expression of the
Hamiltonian, as soon as conditions (A) and (B) are satisfied, $\Phi$
is a viscosity solution of the Hamilton-Jacobi equation
\eqref{HJgen}.

We recall  (see \cite{B}, \cite{CL}, \cite{E}) that a function
$\varphi$ is a viscosity solution for the Hamilton--Jacobi
 equation $H(x, \nabla \varphi(x) )=0$
 iff $\varphi \in C(\bbT)$ and
 \begin{align}
&  H(x ,p)\geq  0 \qquad \forall x \in \bbT, \; p \in D^-\varphi(x)
\,, \label{anno1}
\\
& H(x ,p)\leq  0 \qquad \forall x \in \bbT, \; p \in
D^+\varphi(x)\,, \label{anno2}
\end{align}
 where the superdifferential $D^+ \varphi(x)$ and the subdifferential
$D^- \varphi (x)$ are defined as
\begin{align}
& D^+ \varphi (x)= \Big\{p \in \bbR\,:\, \limsup _{y \rightarrow x}
 \frac{\varphi(y)-\varphi(x) -p (y-x) }{|y-x|}\leq 0  \Big\}\,, \label{sveglia1}\\
& D^- \varphi(x)= \Big\{ p \in \bbR \,:\,\liminf  _{y \rightarrow x}
\frac{\varphi(y)-\varphi(x) -p  (y-x) }{|y-x|}\geq  0 \Big\}\,.
\label{sveglia2}
\end{align}
As explained in \cite{FW} for the diffusion and in \cite{FGR2}
 for the PDMP, Hamilton--Jacobi equations appear in a natural way in
 problems related to the computation of the quasipotential.
In the case of the diffusive model the associated Hamiltonian is
\begin{equation}
\label{HJ} H(x,p):= p(p-F(x) ) \, , \;\;(x,p) \in \bbT\times \bbR\,.
\end{equation}
It is easy to check that the Hamiltonian in \eqref{HJ} satisfies
hypotheses $(A)$ and $(B)$ of Theorem \ref{genteo}. For PDMPs the
corresponding Hamiltonian has a more complex structure and is
written in \cite{FGR2} page 298. Also in this case it is possible to
check that hypotheses $(A)$ and $(B)$ are satisfied.

\subsection{Outline of the paper}
In Section \ref{fufi} we recall the definition of the function $W$
given in \cite{FW}. This definition consists of two optimization
problems, whose detailed solution is described in Theorem
\ref{ruspabis} in Section \ref{otto}. In Section \ref{fatica} we
prove Theorem \ref{ruspabis} and part (ii) of Theorem \ref{ruspa}.
In Section \ref{barbidu} we prove part (i) of Theorem \ref{ruspa},
while in Section \ref{urlone} we prove Theorem \ref{genteo}.

\section{Definition of the function $W$ given in
\cite{FW}}
% and minimizers of the involved variational problems}
\label{fufi}
 For the reader's convenience and in order to
set the notation for further developments, in this section we
briefly recall the definition of the function $W$ given in Chapter 6
of \cite{FW}. It is convenient to set $F(x)=- 2 b(x)$ and to work
with $F$ instead of the   field $b$ entering in the definition of
the diffusion \eqref{agostino}. We use the same notation of Theorem
\ref{ruspa}. Moreover, we write $F_-$ and $F_+$ for the positive and
negative part of $F$, respectively. This means that $F_- (x)=
|F(x)\wedge  0 |$ and $F_+ (x) = F(x) \lor 0$.   In what follows we
identify the torus $\bbT$ with the 1D circle and we write $\p$ for
the canonical projection $\p: \bbR \rightarrow \bbR/\bbZ = \bbT$. In
particular, the clockwise orientation of the 1D circle corresponds
to the orientation of the path $\p (x)$ as $x $ goes from $0$ to
$1$. In order to avoid confusion in the rest of the paper  we will
not identify $\bbT$ with the interval $[0,1)$.  Finally, we recall
that we think of $F$ both as function  on the torus and  as periodic
function on $\bbR$ with unit period and that, given $x \in \bbT$, we
write $\max/\min _{y \in [x,x+1]} ( S(y)-S(x) )$ for the quantity
$\max/\min  _{\bar y \in [\bar x,\bar x+1]} ( S(\bar y)-S(\bar x) )$
where $\bar x$ is any point in $\bbR$ such that $\p(\bar x)=x$.

\smallskip

Since the function $S$ is constant on each connected component $C$
of $\{ x \in \bbR\,:\, F(x)=0\}$, we denote by $S(C)$ the constant
value $S(x)$, $x \in C$.  Recall the assumption for
\eqref{nonnabruna} that  the closed set $\{ x \in \bbT\,:\,
F(x)=0\}$ has a finite number of connected components.
 We say that
 a connected  component $[a,b]$ is stable if for some $\e>0$ it holds
$F(s)<0$ for all $s \in [a-\e, a)$ and $F(s)>0$ for all $s \in
(b,b+\e ]$. We say that  $[a,b]$ is totally unstable if for some
$\e>0$ it holds $F(s)>0$ for all $s \in [a-\e, a)$ and $F(s)<0$ for
all $s \in (b,b+\e ]$. If the connected components of $\{x \in
\bbT\,:\, F(x)=0\}$ are given by isolated points, then the stable
and unstable ones are obtained by  the canonical  projection $\p:
\bbR \rightarrow \bbT$ of  the points of local minimum and local
maximum for $S$, respectively.

\smallskip

 We call $K_1,K_2, \dots , K_\ell$ the stable connected components
of $\{ x\in \bbT\,:\,   F(x)=0\}$ labeled  in clockwise way, i.e.
$K_i$ and $K_{i+1}$ are nearest--neighbors and $K_{i+1}$ follows
$K_i$ clockwise. Moreover, for generic $i\in \bbZ$, we denote by
$K_i$ the component $K_j$, such that $1 \leq  j \leq \ell$ and
$i\equiv j$ in $\bbZ/ \ell \bbZ$. We observe that between $K_i$ and
$K_{i+1}$ (with respect to clockwise order) there exists only one
connected component of $\{x\in\bbT\,:\, F(x)=0\}$ totally unstable.
We call $A_i$ such connected component.

\smallskip

Given two distinct points $x,y \in \bbT$ we write $\g_{x,y}^+$  and
$\g_{x,y}^-$ for  the  unoriented paths in $\bbT$ connecting $x$ to
$y$ clockwise and  anticlockwise, respectively. If $x=y$ we set
$\g^+_{x,x}= \g^-_{x,x}=\{x\}$. Note that $\g _{x,y}^-= \g_{y,x}^+$.
Given a generic function $h:\bbT\rightarrow \bbR$ the integrals
$\int _{\g_{x,y}^+} h(s) ds $, $\int _{\g_{x,y}^- } h(s)ds$ are
defined as
\begin{equation}
\int_{\g_{x,y}^+} h(s) ds =
\begin{cases}
 \int _{\bar x}^{\bar y } h(s) ds & \text{ if } \bar x<\bar y\,,\\
 \int _{\bar x}^{\bar y+1} h(s) ds & \text{ if } \bar y<\bar x\,,
\end{cases}
\end{equation}
and
 \begin{equation}
 \int_{\g_{x,y}^-}
h(s) ds =
\begin{cases}
\int _{\bar y}^{\bar x} h(s)  ds & \text{ if } \bar y<\bar x\,,\\
\int _{\bar y}^{\bar x+1} h(s) ds & \text{ if } \bar x<\bar y\,,
\end{cases}
\end{equation}
 where $\bar x, \bar y $ are the only elements in $[0,1)$   such that $\p(\bar
 x)= x$ and $\p(\bar y )=   y$.
%  in the above formulas, given $x \in \bbT$, with some abuse of notation we denote by
%$x$ also the number in $[0,1)$ mapped to  $x$ by  the canonical
% embedding $[0,1) \hookrightarrow \bbT$.

\smallskip

Given $x,y \in \bbT$ we define
\begin{align}
& V_+ (x,y)= \int _{\g^+ _{x,y} } F_+(s)ds\,,\\
& V_-(x,y)= \int _{ \g^- _{x,y}} F_- (s) ds \,,\\
& V(x,y) =  V_+ (x,y)\wedge V_- (x,y) \,. \end{align}  Due to Lemma
3.1 in Section 4.3 in \cite{FW} and the discussion in Section 6.4 in
\cite{FW}, the above function $V(x,y)$ coincides with the function
$V(x,y)$ defined in \cite{FW} at page 161 as well as our definition
of stable component coincides with the one given on page 188 in
\cite{FW}.

\smallskip

Given two distinct connected components $C_1, C_2 $ in
$\{F=0\}\subset \bbT $, the quantities  $V(x,y)$, $V_\pm (x,y)$ with
$x \in C_1$ and $y \in C_2$ do not depend on the particular choice
of the  points $x,y$ and are denoted respectively by $V(C_1,C_2)$, $V_\pm
(C_1, C_2)$. We set $V(C_1,C_1)= V_\pm (C_1,C_1)=0$ (note that
$V(x,y)=0$ for all $x,y \in C_1$, while due to our definition $V_\pm
(x,y)$ typically depend  from $x,y \in C_1$).

\smallskip
%Similarly, given $z \in \bbT$, the quantities $V(x,z)$ and $V_\pm
%(x,z)$ do not depend on $x \in C_1 $ and are denoted respectively by
%$V(C_1,z)$, $V_\pm (C_1, z)$. Same fact for $V(z,C_1)$, $ V_\pm (z,
%C_1)$.

Let us now pass to define the function $W$. If $\{ F=0\}=\emptyset$,
then set $W\equiv 0$. When $\{F=0\}\neq \emptyset$ then necessarily
there is at least one stable connected component. If there is only
one stable connected component, then set $W(K_1)=0 $. Let us now
assume $\ell \geq 2$.

\smallskip

Given an index $i$ in $\{1, \dots, \ell\}$, an oriented graph $g$ is
said to belong to the family $G\{ i \}$ if it is a directed tree
having vertices $\{1,\dots, \ell\}$, rooted at $i$ and pointing
towards the root. This means that
\begin{enumerate}
\item
 $1,\dots, \ell$ are the vertices of  $g$,

\item

 every point $j$ in  $\{1, \dots, \ell\}\setminus \{i\}$  is the
initial point of exactly one arrow,

\item

for any point $j$ in $\{1, \dots, \ell\}\setminus \{i\}$ there
exists a (unique) sequence of arrows leading from it to the point
$i$,

\item

no arrow exits from $i$.
\end{enumerate}

Given a graph $g \in G\{i\}$ the arrow from $m$ to $n$ is denoted by
$m\rightarrow n$.
  Then in \cite{FW} the authors  define
\begin{equation}\label{oche}
 W(K_i) = \min _{g \in G\{i\}} \sum _{(m\rightarrow n)\in g} V(K_m,
 K_n) \,.
 \end{equation}

\begin{Def}\label{trattore} The function $W$ on $\bbT$ is defined in \cite{FW} as
follows: if $\ell=0$ then  $W(x)=0$ for all $x \in \bbT$, otherwise
\begin{equation}\label{anatre}
W(x)=\min _{i\in \{1, \dots, \ell\} } \bigl[ W(K_i) + V(K_i, x )
\bigr]\,.
\end{equation}
\end{Def}
We point out that it holds $W(x)= W(K_i)$ for all $x \in K_i$ as
discussed in  \cite{FW}.
% Moreover, if $\ell=1$ then it holds $W(x)= V(K_1,x)= V

\section{Solution of the variational problems entering in the
definition of $W$}\label{otto}

Given $x,y$ in $\bbT$ we define $\D S(x,y)$ as \begin{equation} \D
S(x,y):= \int _{\g^+_{  x,  y } } F(s) ds\,.
\end{equation}
We note that  $ S(v)-S(u)= \D S(x,y)$ whenever $v \in [u,u+1)$ and
$x=\p(u), y=\p(v)$, $\p$ being the canonical projection
$\p:\bbR\rightarrow \bbT$.
%In addition, we observe that $\D S(x,y)=
%\D S(x',y')$ if $x,x'$ and $y,y'$ belong to the same connected
%components of $\{F=0\}\subset \bbT$.
Given $C,C'$ two distinct connected components of $\{F=0\}\subset
\bbT $, we set $\D S(C,C'):= S(x,y)$ for any  $x\in C, y\in C'$.
Trivially the definition does not depend on the choice of $x,y$.

\medskip

Given disjoint connected subsets $A,B \subset \bbT$ and a point $x
\in \bbT$ we write $A \leq  x \leq  B$ if $x \in \g^+_{a,b}$ for
some $a \in A$ and $b \in B$. We write $A<x \leq B$ if $A \leq x
\leq B$ and $x \not \in A$. Similarly, we define $A\leq x < B$ and
$A< x <B$. Then, it holds $$ \D S(K_i, A_i)= \max \bigl\{ \D
S(x,y)\,:\, K_i \leq x\leq y \leq K_{i+1}\bigr\}\,. $$
(Recall that $A_i$ is the totally unstable connected component of $\{F=0\}\subset \bbT$ between $K_i$ and $K_{i+1}$.)

\medskip

We  introduce  some special graphs $g_{i,j} \in G\{i\}$ with $i,j\in
\{1, \dots, \ell\} $. The simplest  definition is given by Figure
\ref{mayetto}.

\begin{figure}[!ht]
    \begin{center}
       \psfrag{i}[l][l]{$i$}
      \psfrag{j}[l][l]{$j$}
       \psfrag{u}[l][l]{$i+1$}
 \psfrag{l}[l][l]{$i-1$}
  \psfrag{a}[l][l]{$g_{i,j}$}
 \psfrag{b}[l][l]{$g_{i,i}$} \psfrag{c}[l][l]{$g_{i,i-1}$}
 \psfrag{k}[l][l]{$j+1$}
 \includegraphics[width=10cm]{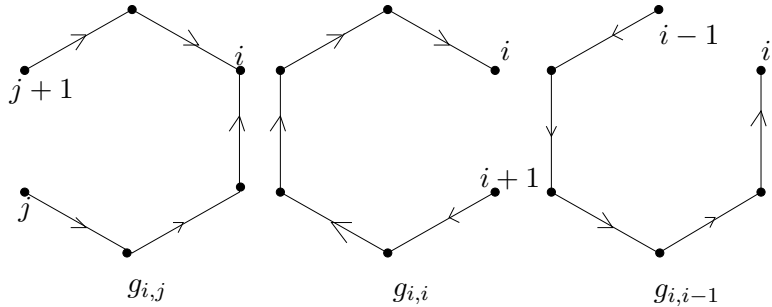}
      \caption{The graphs $g_{i,j}$. Vertices  are labeled clockwise from  $1$ to $ \ell$. }
    \label{mayetto}
    \end{center}
  \end{figure}

 In order to  give a formal definition,  at cost of a
rotation  we can assume that $i=1 \leq j \leq \ell$. Then, $g_{1,j}$
% has vertices $1, \dots, \ell$, it has
%anticlockwise arrows $i+1 \rightarrow i$, $i+2\rightarrow i+1$,...,
%$J\rightarrow J-1$.
has anticlockwise arrows
 $ j\rightarrow j-1$, $ j-1\rightarrow
j-2$,..., $2\rightarrow 1$ and clockwise arrows $j+1\rightarrow
j+2$, $j+2\rightarrow j+3$,..., $\ell \rightarrow \ell+1\equiv 1$.
Note that if $j=1$ the graph $g$ has only clockwise arrows
($2\rightarrow 3$, $3\rightarrow 4$,..., $\ell\rightarrow
\ell+1\equiv 1$), while if $j=\ell$ the graph $g$ has only
anticlockwise arrows ($\ell\rightarrow \ell-1$, $\ell-1 \rightarrow
\ell-2$,..., $2 \rightarrow 1$). As we will show,  in the minimization
problem \eqref{oche} only the graphs of the type $g_{i,j}$ are
relevant.
% This is somehow intuitive for topological reasons by the
%dynamic nature of the Freidlin and Wentzell construction,
%nevertheless we will give a rigorous proof of this fact.

\medskip

 We can finally state our result:

\begin{Th}\label{ruspabis}
The minimizers in \eqref{oche} and \eqref{anatre} can be described
as follows.

(i) Take $\ell \geq 2$.  Fixed $i \in \{1, \dots, \ell\}$, take
$J\in \{1,\dots, \ell\} $ such that
$$
\D S(K_i, A_J)= \max _{y \in [x,x +1]}\bigl( S(y)-S(x )\bigr)$$ for
some (and therefore for all) $x \in K_i$. Then
\begin{equation} W(K_i)= \sum _{ (m\rightarrow n) \in g_{i,J} } V(K_m,K_n) \,,
\end{equation} and for each arrow $m\rightarrow n$ in $g_{i,J}$ it
holds
\begin{equation}
V(K_m,K_n)=
\begin{cases} V_- (K_m, K_{m-1}) & \text{ if } n=m-1
\,,\\
V_+(K_m, K_{m+1}) & \text{ if } n=m+1 \,.
\end{cases}
\end{equation}

\medskip

 (ii)  Take $\ell \geq 1$. Take $ x \in \bbT \setminus \bigl( \cup _{r=1}^\ell K_r\bigr)$. Take $i\in \{1, \dots,\ell\}$ such that $K_i< x < K_{i+1}$.
%  Recall that $A_i$ is the set
%of maximum points of $S$ between $K_i$ and $K_{i+1}$.
Fix $x_i \in K_i$, $x_{i+1}\in K_{i+1}$. Then, fix  $\bar x_i \in
\bbR$ such that $x_i=\p(\bar x_i)$, and afterwards fix $ \bar x,
\bar x_{i+1}\in (\bar x_i, \bar x_i +1]$  such that $x= \p (\bar
x)$, $x_{i+1}= \p (\bar x_{i+1})$.

Then exactly  one of the following cases holds:
\begin{enumerate}

\item $x_i<x\leq  A_i$ and
 $\max _{y \in [\bar x_i, \bar x_i+1] } S(y)= \max_{y \in [\bar x,\bar x+1]} S(y) $;

 \item
$x_i<x\leq  A_i$ and
 $\max _{y \in [\bar x_i, \bar x_i+1] } S(y)< \max_{y \in [\bar x,\bar x+1]} S(y) $;

\item
$A_i\leq x< x_{i+1}$ and
 $\max _{y \in [\bar x_{i+1}, \bar x_{i+1}+1] } S(y)= \max_{y \in [\bar x,\bar x+1]} S(y)
 $;

\item

$A_i\leq x< x_{i+1}$ and
 $\max _{y \in [\bar x_{i+1}, \bar x_{i+1}+1] } S(y)< \max_{y \in [\bar x,\bar x+1]} S(y) $.

\end{enumerate}
Then in cases (1) and (4) it holds \begin{equation}\label{lorenzo1}
W(x)= W(K_i)+ V(K_i,x)\,, \qquad V(K_i,x)= V_+ (K_i,x)\,,
\end{equation}
while in cases (2) and (3) it holds
\begin{equation}\label{lorenzo2}
W(x)= W(K_{i+1})+ V(K_{i+1},x)\,, \qquad V(K_{i+1},x)= V_-
(K_{i+1},x)\,.
\end{equation}

 \end{Th}

\section{Proof of Theorem \ref{ruspa}(ii)  and Theorem \ref{ruspabis}
}\label{fatica}
  Part (ii)  of Theorem \ref{ruspa}  follows easily from the following
 result:

\begin{Pro}\label{biancaneve}
For each $x \in \bbT$
%  given  $\bar x \in \bbR$ such that  $x = \p(\bar x) $,
it holds
\begin{equation}\label{pisolo}
W(x)-\Phi(x)=W(x)+ \max _{y \in [  x,   x+1]} \bigl(S(y)-S(
x)\bigr) = \int _0 ^1 F_+ (s) ds\,,
\end{equation}
In particular, the l.h.s. does not depend on $x$.
\end{Pro}

We divide the proof of the above proposition in several steps.
% From now on we take $\ell \geq 1 $, we will come back to the (very
%trivial) case $\ell=0$   only at the end.

% First
%we prove \eqref{pisolo} for $x \in K_i$, stable connected component
%of $\{F=0\}$.
\begin{Le}\label{cuore}
Given $\ell \geq 2$ and  $i,j\in \{1,\dots , \ell\} $,  consider
$$t(i,j):=V_- (K_j,K_i)+ V_+(K_{j+1},K_i) \,.
$$
Then, \begin{equation}\label{cugini} W(K_i)= \min_{1\leq j \leq
\ell} t(i,j)\,.
\end{equation}
Moreover,  $ t(i,J)= \min _{1\leq j \leq \ell} t(i,j)$ if and only
if
\begin{equation}\label{netti} \D S (K_i, A_J)= \max _{y \in [x,x+1] } \bigl(S(y)-S(x)\bigr)\,, \qquad \forall x \in K_i \,.
\end{equation}
\end{Le}
\begin{proof}
 In order to simplify the notation, without loss of generality
we take $i=1$.  Let us first show that
 \begin{equation}\label{setissima} W(K_1)=\min_{g \in G\{1\}
} \sum_{(m \rightarrow n) \in g} V(K_m,K_n) \leq  \min_{ 1\leq j
\leq \ell} t(1,j)\,.
\end{equation}
  For each $j \in \{1,\dots, \ell\}$, consider the graph $g_{1,j}\in
G\{1\}$ defined in Section \ref{fufi}.
% Take Take $J$ such that
%$t(1,J)= \min_{1\leq j \leq \ell} t(1,j)$. We claim that
%\begin{equation}
%t(1,J)= \sum_{(m\rightarrow n)\in g_J} V(K_m,K_n)\,.
%\end{equation}
%Note that this implies \eqref{setissima}. In order to prove our
%claim, it is enough to show that \begin{align} &  V_-(K_r, K_{r-1})
%\leq V_+ (K_r, K_{r-1})\,, \qquad 2 \leq r \leq J\,,\\
%& V_+(K_r, K_{r+1})\leq V_-(K_r, K_{r+1})\,, \qquad J+1\leq r
%<\ell\,.
%\end{align}
%Suppose for example that for some $r$ with $2 \leq r \leq j$ it
%holds $V_-(K_r, K_{r-1}) >V_+ (K_r, K_{r-1})$. We take $r$ as the
%minimum one. We claim that $t(1, r)< t(1,J)$, thus arriving at a
%contradiction.
If $j \not=1, \ell$ we have
\begin{equation}\label{occhio}
t(1,j)\geq  \sum _{r=2}^{j} V(K_r, K_{r-1})+ \sum _{r=j+1}^{\ell-1}
V(K_r, K_{r+1})= \sum _{(m\rightarrow n)\in g_{1,j}} V(K_m,K_n)\,,
\end{equation}
since
\begin{align}
&  V_-(K_j, K_1) = \sum _{r=2}^{j} V_-(K_r, K_{r-1})\geq \sum
_{r=2}^{j} V(K_r, K_{r-1})\,, \label{squillo1} \\
& V_+(K_{j+1},K_1)= \sum _{r=j+1}^{\ell } V_+(K_r, K_{r+1})\geq
\sum _{r=j+1}^{\ell } V(K_r, K_{r+1})\,. \label{squillo2}
\end{align}
We point out that \eqref{squillo1} holds also for $j=\ell$, while
\eqref{squillo2} holds also for $j=1$.
% This implies that
%  \begin{align*}
% &  t(1,1)= V_+ (K_2,K_1)=\sum _{r=2}^{\ell } V_+(K_r, K_{r+1})\geq \sum _{(m\rightarrow n)\in g_{1,1}} V(K_m,K_n)\,,\\
% & t(1,\ell)= V_-( K_\ell,K_1)= \sum _{r=2}^{\ell} V_-(K_r,
% K_{r-1})\geq \sum _{(m\rightarrow n)\in g_{1,\ell}} V(K_m,K_n)\,.
% \end{align*}
As a  consequence,  \eqref{occhio} is valid also for $j=1,\ell$ and
this readily implies  \eqref{setissima}.
\medskip

 Let us now prove that \eqref{setissima} remains valid with opposite inequality.
  To this aim, we take a generic graph $g \in G\{1\}$ and fix $x_n\in K_n$ for each stable connected
  component $K_n$. Given an arrow
  $m\rightarrow n$, we define  $\s(m\rightarrow n)\in\{-,+\}$ as
$$
 \s(m\rightarrow n )=  \begin{cases}
 +  & \text{ if } V_-(K_m,K_n ) \geq  V_+(K_m,
K_n )\,, \\
- & \text{ if } V_-(K_m,K_n )<  V_+(K_m, K_n )\,.
\end{cases}
$$
Note that
\begin{equation}\label{brumsporco}
 V(K_m, K_n)=V_{\s(m\rightarrow n)} (K_m,K_n)=
  \int _{\g^{\s(m\rightarrow n) } _{x_m, x_n}} F_{\s(m\rightarrow n)} (s)ds\,.
\end{equation}

\medskip
Given $g \in G\{1\}$, let us call
\begin{equation}
\Delta(g):=\sum_{(m\to n)\in g}\sum_{j \in \{1, \dots
,\ell\}\setminus   \left\{m,n\right\}}\chi\left( x_j \in
\gamma_{x_m,x_n}^{\sigma(m\to n)} \right)\,, \label{cont}
\end{equation}
where $\chi$ denotes the characteristic function ($\chi (A)$ equals
$1$ if the condition $A$ is satisfied and zero otherwise). We denote
by $G^0\left\{1\right\}$ the subset of $G\left\{1\right\}$
containing the elements $g$ satisfying $\Delta(g)=0$, i.e. such that
for any $(m\to n)\in g$ and for any $j\ne m,n$, it holds $x_j\not
\in \gamma_{x_m,x_n}^{\sigma(m\to n)}$. We next show that
$G^0\left\{1\right\}$ is not empty and that  we can restrict to
$G^0\left\{1\right\}$ the first  minimum appearing in equation
\eqref{setissima} (we assume $\ell>2$ otherwise this statement is
obviously true). To this aim, if $\D(g)\geq 1$ we fix $m,n,j$ such
that   $x_j \in \gamma_{x_m,x_n}^{\sigma(m\to n)}$ and construct a
new graph $g'$ satisfying  the following properties:
\begin{itemize}

\item[(i)]  $g'\in G\left\{1\right\}$,

\item[(ii)] $\Delta(g')\leq \Delta(g)-1$,

\item[(iii)] $\sum_{(r \rightarrow s) \in g'} V(K_r,K_s)\leq \sum_{(r \rightarrow s) \in g}
V(K_r,K_s)$.

\end{itemize}
We need to distinguish two cases:
\begin{itemize}

\item[(C1)] the unique path in $g$ from $j$ to $1$ does not contain the arrow
  $m\rightarrow n$,
\item[(C2)]
 the unique path in $g$ from $j$ to $1$  contains  the arrow
  $m\rightarrow n$.
\end{itemize}
In case (C1) the graph $g'$ is obtained from $g$ by
 removing the edge $m\to n$  and adding the edge $m\to
j$. In case (C2) we call $j'$  the arriving point of the unique
arrow in $g$ exiting from $j$ (note that in this case necessarily
$j\neq 1$), then we remove from $g$ the arrows $m\to n$ and $j \to
j'$ and add the arrows $m\to j$ and $j \to n$. See Figure \ref{transform}.

%kuka

\begin{figure}[!ht]
    \begin{center}
      \psfrag{A}[l][l]{Case (C1)}
      \psfrag{b}[l][l]{$x_1$}
      \psfrag{c}[l][l]{$x_n=x_2$}
      \psfrag{d}[l][l]{$x_3$}
      \psfrag{e}[l][l]{$x_j=x_4$}
      \psfrag{f}[l][l]{$x_m=x_5$}
      \psfrag{g}[l][l]{$x_6$}
      \psfrag{h}[l][l]{$x_7$}
      \psfrag{i}[l][l]{$x_1$}
      \psfrag{l}[l][l]{$x_n=x_2$}
      \psfrag{m}[l][l]{$x_3$}
      \psfrag{n}[l][l]{$x_j=x_4$}
      \psfrag{o}[l][l]{$x_m=x_5$}
      \psfrag{p}[l][l]{$x_6$}
      \psfrag{q}[l][l]{$x_7$}
 \psfrag{R}[l][l]{Case (C2)}
      \psfrag{s}[l][l]{$x_1$}
      \psfrag{t}[l][l]{$x_n=x_2$}
      \psfrag{u}[l][l]{$x_3$}
      \psfrag{v}[l][l]{$x_j=x_4$}
      \psfrag{z}[l][l]{$x_m=x_5$}
      \psfrag{x}[l][l]{$x_{j'}=x_6$}
      \psfrag{y}[l][l]{$x_7$}
      \psfrag{B}[l][l]{$x_1$}
      \psfrag{C}[l][l]{$x_n=x_2$}
      \psfrag{D}[l][l]{$x_3$}
      \psfrag{E}[l][l]{$x_j=x_4$}
      \psfrag{F}[l][l]{$x_m=x_5$}
      \psfrag{G}[l][l]{$x_{j'}=x_6$}
      \psfrag{H}[l][l]{$x_7$}
 \includegraphics[width=13cm]{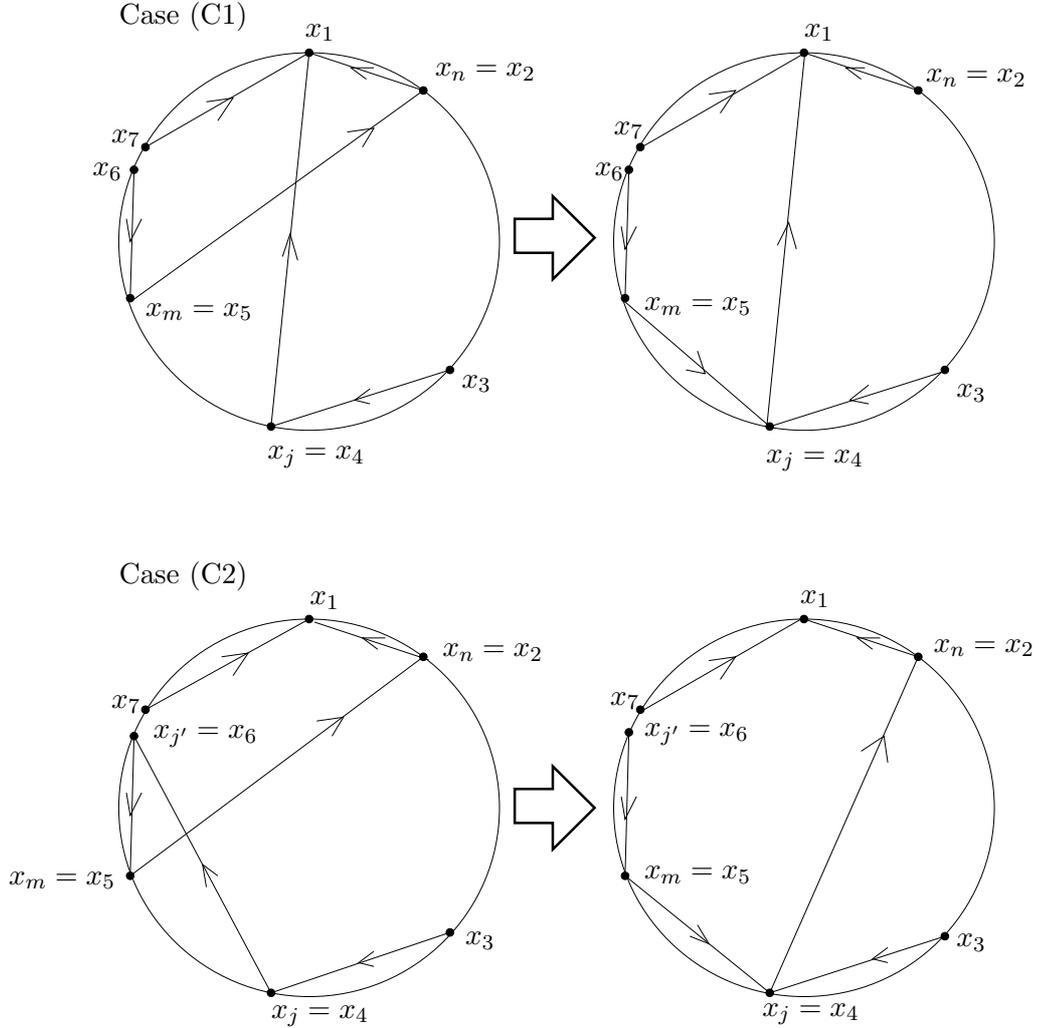}
      \caption{Transformation $g\to g'$  in cases (C1) and (C2), with $m=5$, $n=2$, $j=4$ and
$\s(m\to n)=-$. In the above picture, we have identified graphs on
$\{1,\dots, \ell\}$ with graphs on $\{x_1, \dots, x_\ell\}$. }
    \label{transform}
    \end{center}
  \end{figure}

\medskip

 Let us now show that
$g'$ satisfies the properties (i), (ii) and (iii).

\smallskip

$\bullet $ Proof of (i). Trivially, in all cases properties (1), (2)
and (4) in the definition of $G\{i\}$ given in Section \ref{fufi}
are satisfied. We focus on property (3).

Consider first case (C1). Removing the arrow $m\to n$ the graph $g$
splits into two connected  components. One component coincides with the
vertices whose unique path towards the root $1$ in $g$ contains the
arrow $m\to n$ (in particular $m$ belongs to this component), the
other component coincides with the vertices whose unique path
towards the root $1$ in $g$ does not use the arrow $m\to n$ (in
particular both $j$ and $n$ belong to this component). Both
components are trees. Moreover, by definition of $g$, the component
containing $m$ is oriented towards $m$. Likewise the other component
is a directed tree oriented towards its root $1$. If we add the
arrow $m\to j$ the graph that we obtain is connected and it is a
tree.  Using the orientations of the two merged components we easily
obtain that it is oriented towards its root $1$.

Consider now case (C2). By arguments similar to the
previous case we have the following. Removing the arrows $m\to n$
and $j\to j'$ we obtain three connected components. One is a
directed tree oriented towards its root $j$, one is a directed tree
oriented towards its root $m$ and one is a directed tree containing
$n$ and oriented towards its root $1$. Adding the arrows $m\to j$
and $j\to n$ we obtain a directed tree oriented towards the root
$1$. See \cite{BB} for the basic characterizing properties of trees
that we implicitly used in the proof.

\smallskip

\noindent $\bullet$ Proof of  (ii).  We first consider case (C1). We
have
$$ \D(g')= \D (g) -\sum_{r\neq m,n}\chi\left(x_r \in \gamma_{x_m,x_n}^{\sigma(m\to n)}
\right)+\sum_{r\neq m,j}\chi\left(x_r \in
\gamma_{x_m,x_j}^{\sigma(m\to j)} \right)\,.
$$
Since it must be  $\sigma(m\to j)=\sigma(m\to n)$ and $x_n\not
\in\gamma_{x_m,x_j}^{\sigma(m\to j)}$, we conclude that
\begin{equation}\label{maddalina}
\sum_{r\neq m,j}\chi\left(x_r \in \gamma_{x_m,x_j}^{\sigma(m\to j)}
\right)\leq \sum_{r\neq m,n}\chi\left( x_r \in
\gamma_{x_m,x_n}^{\sigma(m\to n)} \right)-1\,,
\end{equation}
thus implying (ii). Let us now consider case (C2). We have
\begin{multline}\label{igorstra} \D(g')= \D (g) - \sum_{r\neq m,n}\chi\left(x_r \in
\gamma_{x_m,x_n}^{\sigma(m\to n)} \right) -\sum_{r\neq
j,j'}\chi\left(x_r \in \gamma_{x_j,x_{j'}}^{\sigma(j\to j' )}
\right) \\ + \sum_{r\neq m,j}\chi\left(x_r \in
\gamma_{x_m,x_j}^{\sigma(m\to j)} \right) + \sum_{r\neq
j,n}\chi\left(x_r \in \gamma_{x_{j},x_n}^{\sigma(j\to n)} \right)\,.
\end{multline}
Again it must be $\sigma(m\to j)=\sigma(m\to n)= \s(j \to n)$.
Therefore the last two terms above equal
$$
\sum_{r\neq m,j,n}\chi\left(x_r \in \gamma_{x_m,x_n}^{\sigma(m\to
j)} \right) =\sum_{r\neq m,n}\chi\left(x_r \in
\gamma_{x_m,x_n}^{\sigma(m\to n)} \right)-1\,.
$$ This identity together with \eqref{igorstra} leads to (ii).
\smallskip

\noindent  $\bullet$ Proof of (iii). We first consider case (C1).
Then we have \begin{multline} \sum _{(r\to s) \in g } V(K_r, K_s) -
\sum_{ (r\to s) \in g'} V(K_r, K_s) = V(K_m, K_n)- V(K_m, K_j)=
\\\int_{\gamma_{x_m,x_n}^{\sigma(m\to n)}}F_{\s(m\rightarrow n)}
(s)ds-\int_{\gamma_{x_m,x_j}^{\sigma(m\to j)}}F_{\s(m\rightarrow j)}
(s)ds
 \,.
 \end{multline}
Since $\s (m \to n)= \s (m \to j)$ the last difference must be
nonnegative.

In case (C2) we have \begin{multline*}
 \sum _{(r\to s) \in g }
V(K_r, K_s) - \sum_{ (r\to s) \in g'} V(K_r, K_s) = \\
V(K_m, K_n)+
V(K_j,K_{j'})- V(K_m, K_j)-  V(K_j,K_n)\,.
\end{multline*}
The last difference must be positive  since $\sigma(m\to
j)=\sigma(m\to n)= \s(j \to n)$ and therefore  $ V(K_m, K_j)+
V(K_j,K_n)= V(K_m, K_n)$.

\medskip

We have now proved our claim concerning the new graph $g'$.  This
claim trivially implies  that, starting from any initial graph
belonging to $G\left\{1\right\}$, with a finite number of iterations
of the above procedure we end with a graph belonging to
$G^0\left\{1\right\}$. In particular $G^0\left\{1\right\}$ is not
empty. Moreover from property (iii) we have that
$$
\min_{g\in G^0\left\{1\right\}}\sum_{(r \rightarrow s) \in g}
V(K_r,K_s)= \min_{g\in G\left\{1\right\}}\sum_{(r \rightarrow s) \in
g} V(K_r,K_s)
$$
Clearly if $g\in G^0\left\{1\right\}$ then it can contain only
arrows of the type $r\to r+1$ or $r\to r-1$. Moreover if it contains
$r \to r+1 $ then necessarily $V(K_r,K_{r+1})=V_+(K_r,K_{r+1})$; if
it contains $(r\to r-1)$ then necessarily
$V(K_r,K_{r-1})=V_-(K_r,K_{r-1})$. This implies that if $g\in
G^0\left\{1\right\}$ then $g=g_{1,j}$ for some $j$ and moreover
$$
\sum_{(r \rightarrow s) \in g} V(K_r,K_s)=t(1,j)\,.
$$
Summarizing we have
\begin{eqnarray*}
& &\min_{g\in G\left\{1\right\}}\sum_{(r \rightarrow s) \in g}
V(K_r,K_s)= \min_{g\in G^0\left\{1\right\}}\sum_{(r \rightarrow s)
\in g} V(K_r,K_s)\\
& &=\min_{\left\{j:\ g_{1,j}\in G^0\left\{1\right\}\right\}}\sum_{(r
\rightarrow s) \in g} V(K_r,K_s)=\min_{\left\{j:\ g_{1,j}\in
G^0\left\{1\right\}\right\}}t(1,j)\geq \min_{1\leq j\leq
\ell}t(1,j)\,,
\end{eqnarray*}
that is \eqref{setissima} with the reversed inequality.

\medskip

It remains to prove the last statement in Lemma \ref{cuore}. Again,
we take $i=1$ for simplicity of notation. We fix two distinct
indices
 $j,J$ in $\{1,2, \dots, \ell\}$ and we fix  $a_j \in A_j$ and $a_J\in
 A_J$.
 If $1\leq j <J$ then it holds
$$ V_-(K_J, K_j)= \int _{\g ^+_{a_j, a_J} } F_- (s)ds \,, \qquad
 V_+(K_{j+1}, K_{J+1} )= \int _{\g ^+_{a_j, a_J} } F_+ (s)ds\,.
 $$
 Therefore we can write
\begin{multline}
t(1,J)-t(1,j)= V_- (K_J, K_j)- V_+(K_{j+1}, K_{J+1} )=\\
- \int_{\g^+_{a_j, a_J}} F(s)ds= -\D S(a_j, a_J)= -\D S(A_j, A_J)\,.
\end{multline}
Similarly, if $J <j \leq \ell$, it holds $t(1,J)-t(1,j)= \D
S(A_J,A_j)$. In particular, an index $J$ realizes the  $ \min
_{1\leq j \leq \ell} t(1,j)$ if and only if $\D S(A_j, A_J)\geq  0$
for all  $j$ such that $1\leq j <J$ and $\D S(A_J,A_j)\leq 0$ for
all $j$ such that $J <j \leq \ell$. These inequalities read as
follows:
 considering $S$ on the interval $[x_0,x_0+1]$ such that $\p(x_0) \in
 K_1$, the highest local maximum (and therefore the maximum)  of $S$ is attained at all points $y
 \in [x_0,x_0+1]$ such that $\p(y) \in A_J$.
  This coincides with the characterization \eqref{netti}.
\end{proof}

\begin{Rem}\label{pancetta} Fixed $i \in \{1,
\dots, \ell\}$, set
\begin{align*}
& N_i^+=\{ j\in \{1, \dots, \ell\}\,:\, V_+(K_j,K_i)\leq
V_-(K_j,K_i)
\}\,,\\
& N_i^-=\{ j\in \{1, \dots, \ell\}\,:\, V_+(K_j,K_i)> V_-(K_j,K_i)
\}\,.
\end{align*}
It is simple to check that there exists $a \in \left\{0,\dots,
l-1\right\}$ such that
\begin{align*}
& N_i^+=\left\{i-m\,:\, \ 0\leq m\leq a\right\}\,,\\
& N_i^-=\left\{i+m\,:\,  1,\leq m\leq l-a-1\right\}\,,
\end{align*}
One could ask if the index $J$ of Lemma \ref{cuore} can be
characterized as $J=i+l-a-1$ or $J=i+l-a$.  It is easy to check that
this simple characterization cannot hold by drawing suitable
functions $S$.
\end{Rem}

In what follows, in order to make the discussion more intuitive, it
is convenient to use   geometric arguments. To this aim we fix some
language. We call {\sl slope} of the function $S$ its graph
restricted to intervals $I$ of the form $I=[\bar x_i, \bar a_i]$ or
$I=[\bar a_i, \bar x_{i+1} ]$, where $\bar x_i, \bar x_{i+1}, \bar
a_i$ are points in $\bbR$ such that $\p (\bar x _i)\in K_i, \p(\bar
x_{i+1})\in K_{i+1}, \p(\bar a_i)\in A_i$ and $\bar a_i-\bar x_i
<1$, $\bar x_{i+1} -\bar a_i<1$. If $I=[\bar x_i, \bar a_i]$, then
the slope is increasing and its height is set equal to $S(\bar a_i)-
S(\bar x_i)$; if $I=[\bar a_i, \bar x_{i+1} ]$, then the slope is
decreasing and its height is set equals to  $S(\bar a_i)- S(\bar
x_{i+1})$.

\begin{Le}\label{dondolo}
Identity \eqref{pisolo} holds for all $x \in K_i$, $1\leq i \leq
\ell$, $\ell \geq 1$. \end{Le}
\begin{proof}
If $\ell=1$, then the thesis is trivial. Suppose that $\ell \geq 2$ and see Figure \ref{martin1}.
Without loss assume that $i=1$. Fix points $\bar x_1<\bar a_1< \bar
x_2 < \cdots < \bar x_\ell<\bar a_\ell<\bar x_1+1$ such that $\p(\bar
x_j)\in K_j$ and  $\p(\bar a_j)\in A_j $.
Take $J$ as in Lemma \ref{cuore} such that $W(K_1)= t( 1,J)$.

\begin{figure}[!ht]
    \begin{center}
       \psfrag{a}[l][l]{$\bar x_1$}
      \psfrag{b}[l][l]{$\bar a_1$}
       \psfrag{c}[l][l]{$\bar x_2$}
 \psfrag{d}[l][l]{$\bar a_2$}
  \psfrag{e}[l][l]{$\bar x_3$}
 \psfrag{f}[l][l]{$\bar a_3$}
 \psfrag{g}[l][l]{$\bar x_4$}
 \psfrag{h}[l][l]{$\bar a_4$}
 \psfrag{i}[l][l]{$\bar x_1 +1 $}
 \psfrag{x}[l][l]{$+$}
\psfrag{x1}[l][l]{$-$}
\psfrag{x2}[l][l]{$+$}
 \psfrag{x3}[l][l]{$-$}
 \psfrag{x4}[l][l]{$+$}
\psfrag{a1}[l][l]{$+$}
\psfrag{a2}[l][l]{$+$}
 \psfrag{y}[l][l]{$+$}
 \includegraphics[width=12cm]{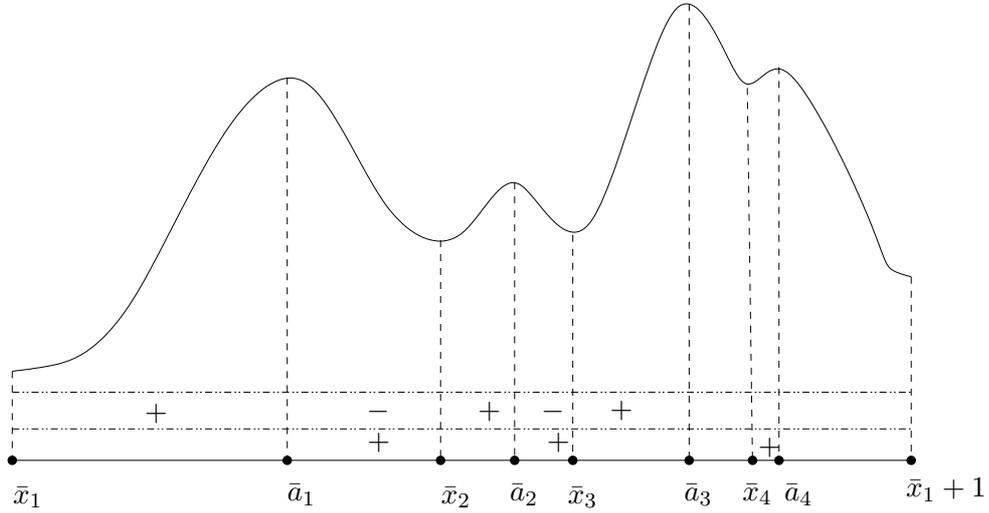}
      \caption{Proof of Lemma \ref{dondolo} with $i=1$, $J=3$, $\ell=4$. The first line of signs illustrates   how to sum slope heights to get $\max _{y\in [\bar x_1, \bar x_1+1]}\bigl( S(y)
-S(\bar x_1) \bigr)$. The second line of signs   illustrates   how to sum slope heights to get $W(K_1)$.}
    \label{martin1}
    \end{center}
 \end{figure}

Then $\max _{y\in [x_1, x_1+1]} \bigl(S(y)- S(x_1) \bigr)$, by
\eqref{netti} in Lemma \ref{cuore}, equals the sum of the heights of
the slopes associated to  $[\bar x_1, \bar a_1], \;[\bar a_1, \bar
x_2],\; \dots, [\bar x_2, \bar a_2],\dots, [\bar x_J, \bar a_J]$
with alternating signs $+,-,+,\dots,+$. On the other hand, by Lemma
\ref{cuore} again, $W(K_1)$ equals the sum of the heights of the
slopes associated to $[\bar a_1, \bar x_2], \dots, [\bar a_{J-1},
\bar x_J]$ and $[\bar x_{J+1},\bar a_{J+1}], \dots, [\bar x_\ell,
\bar a_\ell]$. Hence, $W(K_1)+\max _{y\in [x_1, x_1+1]} \bigl(S(y)-
S(x_1) \bigr)$ simply equals the sum of the heights of the
increasing slopes in $[\bar x_1, \bar x_1 +1]$, i.e. $\int _0^1
F_+(s)ds$.
\end{proof}

\begin{Le}\label{nebbia}
Suppose $\ell \geq 1$. Then, given $x\in \bbT$ and  $i \in
\{1,\dots, \ell\}$, it holds
\begin{equation}\label{mare}
W(K_i)+V_\pm (K_i,x) \geq \int_0^1 F_+(s)ds -\max_{y \in [x,x+1]}
\bigl(S(y)-S(x)\bigr) \,,
\end{equation}
where $V_{\pm}(K_i,x):=0$ if $x\in K_i$ and
$V_{\pm}(K_i,x):=V_{\pm}(x_i,x)$ for any $x_i\in K_i$  if $x\not \in
K_i$.
\end{Le}
\begin{proof}
If $x \in K_i$ we have nothing to prove due to Lemma \ref{dondolo}.
We assume $x \not \in K_i$ and we  fix $x_i \in K_i$, and also $\bar
x , \bar x _i \in \bbR$ such
 that $\p (\bar x )=x$, $\p( \bar x_i ) =x_i$ and $|\bar x -\bar
 x_i|<1$.
 % and $\bar x _i < \bar x <
 %\bar x_i+1$.
 Then, substituting $W(K_i)$ by means of Lemma
\ref{dondolo}, we get that \eqref{mare} reads
\begin{equation}\label{acqua}
 V_\pm (x_i,x) \geq  \max
_{y \in [\bar x_i,\bar x_i+1] }\bigl(S(y)-S(\bar x_i)\bigr)- \max
_{y \in [\bar x,\bar x+1]} \bigl(S(y)-S(\bar x)\bigr)\,.
% V_\pm (x_i,x) \geq S(\bar x)- S(\bar x_i)+ \max _{y \in [\bar x_i,\bar x_i+1] }S(y)- \max _{y \in [\bar x,\bar x+1]}  S(y)\,.
\end{equation}
We give the proof for $V_+ (x_i,x)$. The other case is completely
similar (indeed, specular).
 We can
always choose $\bar x$ and $\bar x_i$ such that $\bar x_i <\bar x <
\bar x_i+1$. Then we can bound \begin{equation}\label{pecora}
 V_+
(x_i, x) =\int _{\bar x_i}^{\bar x} F_+(s) ds \geq \max _{y\in [\bar
x_i, \bar x ] }\int _{\bar x_i}^y F(s) ds = \max _{y\in [\bar
x_i,\bar x]} \bigl( S(y)-S(\bar x_i)\bigr)\,.
\end{equation}
Hence, to conclude it is enough to show that the last member in
\eqref{pecora}  bounds from above the r.h.s. of \eqref{acqua}. This
is equivalent to the inequality
$$
\max _{y\in [\bar x_i,\bar x]} S(y) \geq \max _{y \in [\bar x_i,\bar
x_i+1] }S(y)-\max _{y \in [\bar x,\bar x+1]}S(y) + S(\bar x)\,.$$ If
the l.h.s. equals the first term in the r.h.s., then the inequality
is obviously verified. Otherwise, it must be $\max _{y \in [\bar
x_i,\bar x_i+1] }S(y)<\max _{y \in [\bar x,\bar x+1]}S(y)$ and the
conclusion becomes trivial.
\end{proof}

\begin{Le}\label{bici_wander} Suppose $\ell \geq 1$.
Take $ x \in \bbT \setminus \bigl( \cup _{r=1}^\ell K_r\bigr)$. Take
$i\in \{1, \dots,\ell\}$ such that $K_i< x < K_{i+1}$.
%  Recall that $A_i$ is the set
%of maximum points of $S$ between $K_i$ and $K_{i+1}$.
Fix $x_i \in K_i$, $x_{i+1}\in K_{i+1}$ (if $\ell=1$ take
$x_i=x_{i+1}$). Then, fix $\bar x_i \in \bbR$ such that $x_i=\p(\bar
x_i)$, and afterwards fix $ \bar x, \bar x_{i+1}\in (\bar x_i, \bar
x_i +1]$ such that $x= \p (\bar x)$, $x_{i+1}= \p (\bar x_{i+1})$.

 Then exactly one of the
four cases (1),...,(4) mentioned in Theorem \ref{ruspabis} (ii) holds. Moreover,
in cases (1) and (4) it holds
\begin{equation}\label{torta1}
V_+(K_i,x)+  \max_{ y \in [\bar x, \bar x+1]} S(y)-S(\bar x) = \max
_{y\in [\bar x_i,\bar x_i +1]} S(y)-S(\bar x_i)\,,
\end{equation}
while in cases (2) and (3) it holds
\begin{equation}\label{torta2}
V_-(K_{i+1},x)+  \max_{ y \in [\bar x, \bar x+1]} S(y)-S(\bar x)=
\max _{y\in [\bar x_{i+1},\bar x_{i+1} +1]} S(y)-S(\bar x_{i+1}) \,.
\end{equation}
\end{Le}

\begin{proof}
In  case (1) it holds
 $V_+(K_i,x)= S(\bar x)-S(\bar x_i)$ and the check of \eqref{torta1} is
 immediate (see Figure \ref{martin2}). Case (3) is similar.
%. Similarly, in case (3) it holds $V_- (K_{i+1},x)= S(\bar x)-
% S(\bar x_{i+1} )$ and the check of \eqref{torta2} is immediate.
We only need to treat case (2),  since case (4) is specular (take a
reflection at the origin). To this aim,  we fix $a_i \in (\bar x_i,
\bar x_i+1]$ such that $\p(a_i)\in A_i$. Then, due to the definition
of case (2) (see Figure \ref{martin2}), it must be
$ \max _{y \in [\bar x,\bar x+1]}S(y) =S(\bar x +1)$ and
$ \max _{y \in [\bar x_{i+1},\bar x_{i+1}+1]}S(y)=S(\bar a_i+1)$.
In particular, \eqref{torta2} is equivalent to
$$
V_-(K_{i+1},x)+ S(\bar x+1)- S(\bar x)= S(\bar a_i+1)- S(\bar
x_{i+1})\,.
$$
Since $S(\bar x+1)- S(\bar x)= S(\bar a_i+1)- S(\bar a_i)= S(1)$,
the above identity is equivalent to $V_-(K_{i+1},x)= S(\bar a_i)-
S(\bar x_{i+1} )$ which is trivially true (see Figure
\ref{martin2}).

\medskip

It remains now to prove that the above four cases are exhaustive.
Suppose for example that $x_i\leq x<A_i$. Then it is trivial to check that          it cannot be $\max _{y
\in [\bar x_i, \bar x_i+1] } S(y)> \max_{y \in [\bar x,\bar x+1]}
S(y) $.
\end{proof}

\begin{figure}[!ht]
    \begin{center}
       \psfrag{a}[l][l]{$\bar x_1$}
      \psfrag{b}[l][l]{$\bar a_1$}
       \psfrag{c}[l][l]{$\bar x_2$}
 \psfrag{d}[l][l]{$\bar a_2$}
  \psfrag{e}[l][l]{$\bar x_3$}
 \psfrag{f}[l][l]{$\bar a_3$}
 \psfrag{g}[l][l]{$\bar x_1+1$}
 \psfrag{h}[l][l]{$\bar a_1+1$}
 \psfrag{i}[l][l]{$\bar x_2 +1 $}
 \psfrag{z}[l][l]{$z$}
 \psfrag{y}[l][l]{$z+1$}
 \includegraphics[width=12cm]{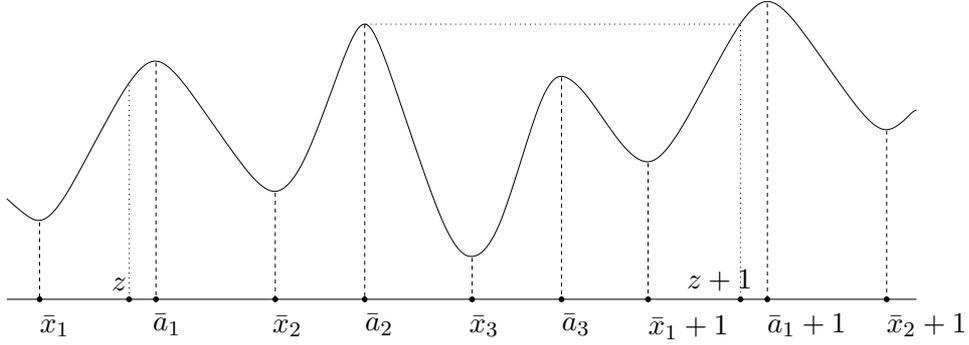}
      \caption{Proof of Lemma \ref{bici_wander} with $i=1$, $\ell=3$. If $\bar x_1 <\bar x\leq z$, then case (1) takes place.
If $z<\bar x< \bar a_1$, then
case (2) takes place.}
    \label{martin2}
    \end{center}
 \end{figure}

We can finally conclude:

\medskip

\noindent {\sl Proof of Proposition \ref{biancaneve} and Theorem
\ref{ruspabis} (ii) }. If $\ell=0$, then  by definition $W\equiv 0$
and we only need to prove Proposition \ref{biancaneve}. We take $
\bar x \in \bbR$ such that $\p(\bar x)= x$ and note that the
function $S$ is monotone. If it is weakly increasing, then
$$\max_{y\in [\bar x,\bar x+1]} \bigl(S(y)-S(\bar x)\bigr)= S(\bar x+1)- S(\bar
x)= \int _0^{1}F(s) ds = \int_0^1  F_+(s) ds $$  which coincides
with \eqref{pisolo}. Similarly, one gets \eqref{pisolo} if $S$ is
weakly decreasing.

\medskip
 Let us restrict now to $\ell\geq 1$.
 If $x \in \cup
_{i=1}^\ell K_i$, we only need to prove Proposition \ref{biancaneve}
and in this case the thesis coincides with Lemma \ref{dondolo}.
 Suppose that $x
\not \in \cup _{i=1}^\ell K_i$. Due to the definition of the
function $W$
 and  inequality
\eqref{mare} in Lemma \ref{nebbia},
$$ W(x) \geq \int_0 ^1 F_+ (s) ds - \max _{y \in [ x, x+1]}
\bigl( S(y)- S( x)\bigr) \,.$$ Moreover, if for some $j \in \{1,
\dots, \ell\}$ and some sign $\s \in \{-,+\}$ it holds
\begin{equation}\label{camillo}
 W(K_j)+ V_\s (K_j,x)=\int_0 ^1 F_+ (s) ds - \max _{y \in [  x, x+1]}
\bigl( S(y)-S( x)\bigr)\,,
\end{equation}
then it must be
\begin{equation}\label{camillo1} W(x)= W(K_j)+ V_\s
(K_j,x)=\int_0 ^1 F_+ (s) ds - \max _{y \in [ x, x+1]} \bigl(
S(y)-S( x)\bigr)\end{equation} and
\begin{equation}\label{camillo2}
V(K_j,x)= V_\s (K_j,x)\,.
\end{equation}
 By Lemma \ref{dondolo} we can rewrite $W(K_j)$ as
$$ W(K_j)= \int_0^1 F_+(s)ds - \max _{ y \in [ x_j,  x_j
+1]}\bigl( S(y) - S( x_j)\bigr) $$ where $x_j   \in K_j$. Hence
\eqref{camillo} reads
\begin{equation}\label{setoso}
V_\s(K_j,x) - \max _{ y \in [ x_j,  x_j +1]}\bigl( S(y) - S(
x_j)\bigr)=- \max _{y \in [  x, x+1]} \bigl( S(y)-S( x)\bigr)\,.
\end{equation}   By Lemma
\ref{bici_wander},  in order to fulfill \eqref{setoso}   it is enough to take $j=i$ and $\s=+$ in cases (1)
and (4) (see \eqref{torta1}), while it is enough to take $j=i+1$
and $\s=-$ in cases (2) and (3) (see \eqref{torta2}). Since as already observed
\eqref{camillo} implies both \eqref{camillo1} and \eqref{camillo2},
this concludes the proof of both Proposition \ref{biancaneve} and Theorem \ref{ruspabis} (ii). \qed

\bigskip
\noindent {\sl Proof of Theorem \ref{ruspa} (ii)}. As already
observed, part (ii) of Theorem \ref{ruspa} is an immediate
consequence of Proposition \ref{biancaneve}. \qed

\bigskip

\noindent {\sl Proof of Theorem \ref{ruspabis} (i)}.
%Recall that
%$\ell\geq 2$ and   that, in order to simplify notation in the proof
%of Lemma \ref{cuore}, we have taken $i=1$ there. In particular, the
%bound \eqref{occhio} can be stated for generic $i \in \{ 1, \dots ,
%\ell\}$ as
%\begin{multline} t(i,j)=
%\sum _{(m\rightarrow n)\in g_{i,j}} \left[V_-(K_m,K_n) \chi (n=m-1)+
%V_+ (K_m, K_n) \chi ( n=m+1) \right]\\ \geq         \sum
%_{(m\rightarrow n)\in g_{i,j}} V(K_m,K_n)\,.
%\end{multline}
If we take $J$ as in Theorem \ref{ruspabis}, by Lemma \ref{cuore} it
must be
\begin{multline}
W(K_i)= t(i,J)=\sum _{(m\rightarrow n)\in g_{i,J}}
\left[V_-(K_m,K_n) \chi (n=m-1)+ V_+ (K_m, K_n) \chi ( n=m+1)
\right]  \\ \geq \sum _{(m\rightarrow n)\in g_{i,J}} V(K_m,K_n) \geq
\min _{g \in G\{i\}}
 \sum
_{(m\rightarrow n)\in g} V(K_m,K_n)= W(K_i)\,.
\end{multline}
In particular all the inequalities in the above expression must be
equalities and this proves the first part of Theorem \ref{ruspabis}.

\section{Proof of Theorem \ref{ruspa} part (i)}\label{barbidu}

The periodicity of $\Phi$ follows by
\begin{eqnarray}
& & \Phi(x+1)=\min_{y\in[x+1,x+2]}(S(x+1)-S(y)
)=\min_{y\in[x+1,x+2]}-\int_{x+1}^yF(s)
ds \nonumber \\
& &=\min_{y\in[x+1,x+2]}-\int_{x}^{y-1}F(t)
dt=\min_{y\in[x,x+1]}-\int_{x}^{y}F(t) dt=\Phi(x)\,.
\end{eqnarray}
In the third equality we used the periodicity of $F$.

\smallskip

Next we show that $\Phi$ is Lipschitz. Due to the periodicity of
$\Phi$ it is enough to show it in $[0,1]$. The function $S$ is
Lipschitz with Lipschitz constant $K:=\max_{x\in \mathbb T}|F(x)|$.
Fix $x<y\in[0,1]$. Then we have
%%%%%%%%%%%%%%%%%%%%%%%%%%%%%%%%%%%%%%%%%%%%
%Fix any $1>\epsilon
%>0$ and consider $\delta\in (0,1) $ such that
%$|S(x)-S(y)|\leq\epsilon$ for any $x,y$ such that $|x-y|\leq\delta$.
%From the Lipschitz property we can fix $\delta=\frac{\epsilon}{K}$.
%We claim that it also holds $|\Phi(x)-\Phi(y)|\leq 2\epsilon$ when
%$|x-y|\leq\delta$. This means That the function $\Phi$ is Lipschitz
%with Lipschitz constant $2\max_{x\in \mathbb T}|F(x)|$. Indeed, we
%can bound
\begin{eqnarray}
& &|\Phi(x)-\Phi(y)|\leq |S(x)-S(y)|+\Big|\max_{z\in
[x,x+1]}S(z)-\max_{w\in [y,y+1]}S(w)\Big|\nonumber \\
& & \leq K|x-y| +\Big|\max_{z\in [x,x+1]}S(z)-\max_{w\in
[y,y+1]}S(w)\Big|\,.\label{lip}
\end{eqnarray}
We now estimate the second term in \eqref{lip}. If the maxima are
achieved respectively in $z^*$ and $w^*$ both belonging to
$[x,x+1]\cap [y,y+1]= [y,x+1]$, then necessarily $S(w^*)=S(z^*)$ and
the second term in \eqref{lip} is zero.
%Let us now   suppose that
%$$ S(z^*)= \min _{z \in [x,y) } S(z) < \min _{z \in [y,x+1]} S(z) \,, \qquad
%\min _{w\in [y,x+1]} S(w) > \min _{w \in (x+1,y+1] } S(w)=S(w^*)\,.
%$$
 Let us suppose that  $S(w^*)>S(z^*)$. Then
necessarily $w^*\in(x+1,y+1]$ and
$$
|S(w^*)-S(z^*)|=  S(x+1)+ S(w^*) -S(z^*)-S(x+1) \leq
S(w^*)-S(x+1)\leq K|x-y|\,.$$ The remaining case can be treated
similarly. Summarizing we have
$$
|\Phi(x)-\Phi(y)|\leq 2K|x-y|\,,
$$
that is $\Phi$ is Lipschitz with Lipschitz constant $2K$.
\smallskip

If $S$ is monotone, then trivially it holds  $\Phi (x)
=\min\left\{0,-S(1)\right\}$ for all $x \in \bbR$.
  When $S$ is periodic, $\max_{y\in [x,x+1]}S(y)$ does not depend on
$x$ and therefore it holds $ \Phi(x)=S(x)-\max_{y\in [0,1]}S(y)$.
Let us suppose now that $S$ is not monotone and it has not period one, i.e. that
  $S(1)=\int_0^1F(s)ds \neq 0$.
  %In particular, we consider only the case
 % $S(1)>0$ (the remaining case can be treated similarly)
  Similarly to our definition on the torus $\bbT$, we say that
$[a,b]  \subset \bbR $ is a totally unstable connected component of
$\{\nabla S=0\}$ if $S$ is constant on $[a,b]$ and there exists
$\e>0$ such that $S(x) < S(a)$ for all $x \in [a-\e,a)  \cup
(b,b+\e]$. Due to the continuity of $F$, the totally unstable connected
components are countable and we enumerate them as $A_j$, $j \in J$.
Due to the assumption that $S$ is not monotone,  the index set $J$
is nonempty.  Below, we write $S(A_j)$ for the value $S(x)$ with $x
\in A_j$.

\smallskip
We write $\Phi (x) = S(x) - \max _{y \in [x,x+1]} S(y)$.  Since
$S(1)= S(x+1)-S(x) \not =0$, then the maximum of $S$ on $[x,x+1]$ is
achieved on
$$
\begin{cases}  \{ y \in [x,x+1]\,:\, y\in  A_j \text{ for some } j\} \cup \{x+1\} &
\text{ if } S(1)>0\,, \\  \{ y \in [x,x+1]\,:\, y\in  A_j \text{ for
some } j\} \cup \{x \} & \text{ if } S(1)<0\,.
\end{cases}
$$
Since $0=S(x)-S(x)$ and $-S(1)= S(x)- S(x+1) $, if we define as in
\eqref{urca}
\begin{equation}\label{cucu} U= \{ x \in \bbR\,:\, \Phi(x) =
S(x)-\max _{y \in [x,x+1]} S(y)\not = \min \{0,-S(1)\} \}\,,
\end{equation}  then for all $x \in U$  the maximum of $S$ on
$[x,x+1]$ is not achieved on $x$ or $x+1$ (one has to distinguish
the cases $S(1)>0$ and $S(1)<0$). Since  $\Phi$ is continuous, we
get that $U$ is an open subset of $\bbR$.

 We define the function
$\Psi$ as $\Psi(x)=\max _{y\in [x,x+1]} S(y)$. Then, $ \Phi(x)= S(x)
- \Psi (x)$. Due to the above observations,  when $x\in U$ we have
$$ \Psi (x)= \max  _{j: A_j \cap
[x,x+1] \not = \emptyset } S(A_j) \,.
$$
Since $\Psi(x)= S(x) -\Phi (x)$ is the sum of two continuous
functions, $\Psi$ is continuous.  We claim that $\Psi$ is constant
on every connected component of $U$. Indeed, consider $(a,b)$ a
connected component of $U$ and suppose there exist $x<y$ in $(a,b)$
such that $\Psi (x) \not = \Psi (y) $. Then the set $\{ \Psi(z)\,:\,
z \in [x,y]\}$ must contain the interval $ \bigl[ \Psi (x) \wedge
\Psi (y), \Psi (x) \lor \Psi (y) \bigr]$, in contradiction with the
fact that, when $x\in U$, $\Psi $ takes value in the countable set
$\{ S(A_j) \,:\, j \in J \} $. This concludes the proof of our
claim, which is equivalent to \eqref{gattone}. By definition of $U$,
one trivially gets \eqref{patroclo}.

\smallskip

Let us prove that $U$ is nonempty and that the interior of
$\bbR\setminus U$ is nonempty . We start with the second claim.
The function $\Phi$ is Lipschitz and consequently it is absolutely
continuous. This implies  that it is almost everywhere
differentiable, its derivative is Lebesgue locally integrable  and
moreover it holds
$$
\int_a^b\nabla \Phi(y)dy=\Phi(b)-\Phi(a)\,,
$$
for any $a,b \in \mathbb R$. From the previous analysis we know
that
  $\Phi$  is differentiable on $U$ where it holds  $\nabla \Phi=
\nabla S $.
% and that  $\Phi$ is differentiable on the interior part
%$A$ of $\mathbb R\setminus U$ where it holds $\nabla \Phi =0$
%(recall that $\Phi= \min \{0, -S(1) \} $ on $\bbR\setminus U$)
Recall that $U$ is open and $\bbR\setminus U$ is closed. If  we
write $A$ for the interior part of $\bbR \setminus U$, then
$B:=(\mathbb R \setminus U)\setminus A$  consists of a countable set
of points and in particular has zero Lebesgue measure. If $A$ was
empty, since $B$ has zero Lebesgue measure, we would conclude that
$\bbR\setminus U= A\cup B$ has zero Lebesgue measure. In particular,
we would get
%\begin{multline}
\begin{equation}
0=\Phi(1)-\Phi(0)=\int_0^1\nabla\Phi(y)dy=\\
\int_{[0,1]\cap U}\nabla S(y)dy=\int _{[0,1]} \nabla S(y) dy = S(1)
\label{cancan}
\end{equation}
in contradiction with the fact that $S(1) \not = 0$.

\smallskip

Let us  show that $U$ is also nonempty. We discuss the case
$S(1)>0$. The case $S(1)<0$ can be treated by similar arguments.
Since $S$ is not monotone, it must have local minima. We claim that
$x \in U$ whenever $x$ is a local minimum point for $S$. Indeed,
since $S(y+1)= S(y)+ S(1)$ for all $y \in \bbR$, also $x+1$ is a
local minimum point for $S$. Since $S$ is not flat (otherwise it
would be monotone), there must be a point $z \in (x,x+1)$ such that
$S(z)> S(x+1)$. This implies that $\max _{ y\in [x,x+1] } S(y) >
S(x+1)$, which is equivalent to $x \in U$, since for $S(1)>0$ the
definition of $U$ reads
$$U= \{ x \in \bbR\,:\, \Phi(x) = S(x)-\max _{y \in [x,x+1]} S(y)\not = S(x)-S(x+1)\}  \}\,. $$

\smallskip

Let us finally prove \eqref{ciampi}, where  only the second identity
is non trivial. Due to the definition \eqref{cucu} of $U$, it must
be $\Phi(x) < \min \{0, -S(1)\}$ for all $x \in U$. On the other
hand, due to \eqref{patroclo}, $\Phi (x) =\min \{0, -S(1)\}$ for all
$x$ in the nonempty set $ \bbR \setminus U$. These considerations
trivially imply \eqref{ciampi}.

\section{Proof of Theorem \ref{genteo}}\label{urlone}

The proof of Theorem \ref{genteo} is based on the following fact:
\begin{Le}\label{grufalo}
 Consider a function
$\varphi:\mathbb T\to \mathbb R$ that satisfies the following
properties:
\begin{itemize}

\item[a)] It is continuous.

\item[b)] There exists an open subset $O\subseteq \mathbb T$ where
it is differentiable and moreover $\nabla\varphi(x)=F(x)$ for any
$x\in O$.

\item[c)] On every connected component of $\mathbb T\setminus O$ it is constant.

\item[d)] Calling $(o^-_j,o^+_j)$, $j\in J$ the countable disjoint maximal
connected components of $O$, it holds $F(o^-_j)\leq 0$ and
$F(o^+_j)\geq 0$ for any $j\in J$.
\end{itemize}
Then,  $\varphi$ is a viscosity solution of \eqref{HJgen} for any
Hamiltonian $H$ satisfying hypotheses (A) and (B) of Theorem
\ref{genteo}.

\end{Le}
\begin{proof}
We first compute the sub-- and superdifferential of $\varphi$ by
simply assuming (a),(b),(c) (assuming (d) would not change much, and
we prefer to make the computation only under (a),(b),(c) since more
instructive).

Trivially,   $\varphi$ is differentiable at $x$ if and only if
$D^+\varphi(x)=D^-\varphi(x)=\left\{a\right\}$ for some value $a \in
\bbR$, and in this case  it holds $a= \nabla \varphi (x)$. In
particular, for any $x\in O$ it must be
$D^+\varphi(x)=D^-\varphi(x)=\left\{F(x)\right\}$, and similarly for
any  $x$ in the interior part of $\mathbb T\setminus O$ it must be
$D^+\varphi(x)=D^-\varphi(x)=\left\{0\right\}$. The non trivial
cases come from $x=o^\pm _j$.

\noindent $\bullet$ If $x=o_j^-$, since
$\varphi(y)-\varphi(x)=\int_x^y F(z) dz$ for $y $ in a small right
neighborhood  of $x$, given $p \in \bbR$ it holds
\begin{equation}\label{grufalo1}
\lim_{y\downarrow
x}\frac{\varphi(y)-\varphi(x)-p(y-x)}{|y-x|}=F(x)-p\,.
\end{equation}
$\bullet$ If $x =o_j^+$  then, by the same argument, we get for any
$p\in \bbR$ that
\begin{equation}\label{grufalo2}
\lim_{y\uparrow
x}\frac{\varphi(y)-\varphi(x)-p(y-x)}{|y-x|}=-F(x)+p\,.
\end{equation}
$\bullet$  If $x=o_j^-$ is not an accumulation point of
$\partial O$, then $\varphi $ is constant on a  small left
neighborhood of $x$ and therefore
\begin{equation}\label{grufalo3}
\lim_{y\uparrow x}\frac{\varphi(y)-\varphi(x)-p(y-x)}{|y-x|}=p\,.
\end{equation}
$\bullet$ Similarly if  $x=o_j^+$ is not an accumulation point of
$\partial O$, then
\begin{equation}\label{grufalo4}
\lim_{y\downarrow x}\frac{\varphi(y)-\varphi(x)-p(y-x)}{|y-x|}=-p\,.
\end{equation}
$\bullet$  If $x=o_j^-$ is  an accumulation point of $\partial O$
and $F(o_j^-)>0$,
 then for any $y$ in  a
 small left neighborhood of $x$ it holds
\begin{equation}\label{protesta}
 0 \geq \varphi (y) -\varphi (x) \geq \int_x^y F(z) dz\,.
 \end{equation}
Indeed, by assumption $F(o_j ^-)>0$ and therefore $F$ is positive on
a small left neighborhood of $x$. Hence, on a small left
neighborhood $\varphi$ is non decreasing. Moreover,  a part a
countable set of points, $\varphi$ is differentiable and $\nabla
\varphi \leq \max \{0, F\}=F$. This leads to the second inequality
in \eqref{protesta}. As a consequence, it holds
\begin{equation}
p\geq \limsup_{y\uparrow
x}\frac{\varphi(y)-\varphi(x)-p(y-x)}{|y-x|}\geq \liminf_{y\uparrow
x}\frac{\varphi(y)-\varphi(x)-p(y-x)}{|y-x|}\geq -F(x)+p\,.
\label{notte2}
\end{equation}
$\bullet$ If $x=o_j^-$ is  an accumulation point of $\partial O$ and
 $F(o_j^-)<0$, by similar arguments  we get for any $y$ in   a small left  neighborhood of $x$  that $$
 \int _x ^y F(z)dz \geq \varphi (y)- \varphi (x) \geq 0\,,$$
and therefore
\begin{equation}
-F(x)+p\geq \limsup_{y\uparrow
x}\frac{\varphi(y)-\varphi(x)-p(y-x)}{|y-x|}\geq \liminf_{y\uparrow
x}\frac{\varphi(y)-\varphi(x)-p(y-x)}{|y-x|}\geq p\,.
\label{nottenotte2}
\end{equation}
$\bullet$
 If $x=o_j^-$ is  an accumulation point of $\partial O$ and
$F(o_j^-)=0$, we can proceed as follows.  On a left neighborhood of
$x$ minus a countable set of points, $\varphi$ is differentiable and
$\nabla \varphi (z) \in \{0, F(z)\}$. This implies that
$$
-\int_x^y|F(z)|dz\geq \varphi(y)-\varphi(x)\geq \int_x^y|F(z)|dz\,.
$$
Since $F(x)=0$, the above bounds trivially imply that
\begin{equation}
\lim _{y \uparrow x} \frac{\varphi(y)-\varphi(x)}{y-x}=0\,.
\label{bellanotte2}
\end{equation}
$\bullet$ Formulas similar to \eqref{notte2}, \eqref{nottenotte2}
and \eqref{bellanotte2} are valid if  $x=o_j^+$ is  an accumulation
point of $\partial O$, $F(o_j^+)>0$, $F(o_j^+)<0$ and $F(o_j^+)=0$
respectively.

\medskip

The above  computations allow us to treat the several possible
cases.
 \medskip

$\bullet$ Case 1: $x=o^-_j=o^+_i$ for some $i,j\in J$.  By
\eqref{grufalo1} and \eqref{grufalo2},  $\varphi$ is differentiable
at $x$ and moreover
$D^+\varphi(x)=D^-\varphi(x)=\left\{F(x)\right\}$.

$\bullet$ Case 2:  $x=o^-_j$ for some $j \in J$ and $x $ is not an
accumulation point of $\partial O$. We claim that
$$
\begin{cases}
 D^-\varphi(x)=[0,F(x)]\,, \;
 D^+\varphi(x)=\emptyset & \text{ if }F(x)>0\,,\\
 D^-\varphi(x)=\emptyset\,, \; D^+\varphi(x)=[F(x),0]\,, & \text{ if
}F(x)<0\,,\\
\nabla \varphi(x)= D^-\varphi(x)=D^+\varphi(x)=\left\{0\right\} &
\text{ if } F(x)=0\,. \end{cases}
$$
Indeed,
%We discuss only the first case: $F(x)>0$. The other cases
%can be treated similarly. Due
due  to \eqref{grufalo1} and \eqref{grufalo3} we conclude that $p\in
D^+\varphi(x)$
 if and only if $ F(x)-p \leq 0 $ and $p\leq 0$, while
$p\in D^-\varphi(x)$ if and only if $ F(x)-p \geq 0 $ and  $p\geq
0$. Then the claim follows by distinguishing on the  sign of $F(x)$.

$\bullet$ Case 3:  $x=o^+_j$ for some $j\in J$ and $x $ is not an
accumulation point of $\partial O$. We claim that
$$
\begin{cases}
 D^-\varphi(x)=\emptyset\,, \;
D^+\varphi(x)=[0,F(x)]\,,
  & \text{
if } F(x)>0\,,\\
 D^-\varphi(x)=[F(x),0]\,, \; D^+\varphi(x)=\emptyset & \text{ if } F(x)<0\,,\\
\nabla \varphi (x)=  D^-\varphi(x)=D^+\varphi(x)=\left\{0\right\} &
\text{ if } F(x)=0\,. \end{cases}
$$
Indeed, by \eqref{grufalo2} and \eqref{grufalo4} it holds: $p \in
D^+\varphi (x) $ if and only if  $-F(x)+p \leq 0 $ and $-p \leq 0$,
while $p\in D^- \varphi(x)$ if and only if $-F(x)+p\geq 0  $ and $-p
\geq 0$.

$\bullet$ Case 4: $x=o^-_j$ and  $x\in\partial O$ is an accumulation
point of $\partial O$. We claim that
$$ \begin{cases}
D^-\varphi(x)\subseteq [0,F(x)]\,,\; D^+\varphi(x)\subseteq
\left\{F(x)\right\} & \text{ if } F(x)>0\,,\\
 D^-\varphi(x)\subseteq \left\{F(x)\right\}\,,\;
D^+\varphi(x)\subseteq [F(x),0] & \text{ if } F(x)<0\,,\\
  D^- \varphi (x)= D^+ \varphi (x)=\{0\}  & \text{ if } F(x)=0\,.
\end{cases}
$$
Indeed, if  $F(x)>0$ then, by \eqref{grufalo1} and \eqref{notte2},
if $p \in D^+ \varphi (x)$ then $F(x)-p\leq 0 $ and $- F(x)+p \leq 0
$, while if $p \in D^- \varphi (x)$ then $F(x)-p \geq 0 $ and $p
\geq 0$. If $F(x)<0$, by \eqref{grufalo1} and \eqref{nottenotte2},
if $p\in D^+ \varphi (x)$ then $F(x)-p\leq 0$ and $ p \leq 0$, while
if $p \in D^- \varphi (x)$ then $F(x) -p \geq 0$ and $-F(x)+p \geq
0$.
 If
$F(x)=0$ then the thesis follows from \eqref{grufalo1} and
\eqref{bellanotte2}.

$\bullet$ Case 5: $x=o^+_j$ and  $x\in\partial O$ is an accumulation
point of $\partial O$. Similarly to Case 4,  one obtains
$$ \begin{cases}
D^-\varphi(x)\subseteq \left\{ F(x) \right\}\,,\;
D^+\varphi(x)\subseteq
[0,F(x)] & \text{ if } F(x)>0\,,\\
 D^-\varphi(x)\subseteq [F(x),0]\,,\;
D^+\varphi(x)\subseteq \left\{ F(x) \right\} & \text{ if } F(x)<0\,,\\
  D^- \varphi (x)= D^+ \varphi (x)=\{0\}  & \text{ if } F(x)=0\,.
\end{cases}
$$

\bigskip

We have now all the tools to   show that $\varphi$ is a viscosity
solution of \eqref{HJgen} at every $x\in \mathbb T$, adding
assumption (d).

First of all we consider a point $x\in \mathbb T$ such that
$F(x)=0$. In this case we proved that $\varphi$ is differentiable at
$x$ and moreover $\nabla \varphi(x)=0$. As  consequence we obtain
that the Hamilton-Jacobi equation is satisfied at $x$ due to the
fact that $D^+\varphi(x)=D^-\varphi(x)=\left\{0\right\}$ and
moreover $H(x,0)=0$ from the hypothesis (B).

We now consider the case $F(x)\neq 0$. If $\varphi$ is
differentiable at $x$ then either
$D^+\varphi(x)=D^-\varphi(x)=\left\{F(x)\right\}$ or
$D^+\varphi(x)=D^-\varphi(x)=\left\{0\right\}$. In both cases by
hypothesis (B) we have that $\varphi$ is a viscosity solution at
$x$. If $\varphi$ is not differentiable at $x$ as a direct
consequence of the previous results, the definitions and all the
assumptions (a), (b), (c) and (d), we have that $\varphi$ is a
viscosity solution at $x$ if the following implications holds:
\begin{equation}
p\in [F(x)\wedge 0,F(x)\vee 0] \qquad \Longrightarrow H(x,p)\leq
0\,, \label{doccia1}
\end{equation}
and
\begin{equation}
p\in (-\infty,F(x)\wedge 0]\cup [F(x)\vee 0,+\infty) \qquad
\Longrightarrow H(x,p)\geq 0\,. \label{doccia2}
\end{equation}
We next show that indeed \eqref{doccia1} and \eqref{doccia2} are
consequences of the hypotheses (A) and (B).

Let us consider first \eqref{doccia1}. Take $p\in [F(x)\wedge
0,F(x)\vee 0]$, then there exists $c\in [0,1]$ such that
$p=c0+(1-c)F(x)$. From hypotheses (A) and (B) we deduce immediately
$$
H(x,p)\leq cH(x,0)+(1-c)H(x,F(x))=0\,.
$$
We discuss now \eqref{doccia2}. Take for example the case $F(x)<0$
and $p\in (-\infty,F(x)]$. Consider an arbitrary $w\in(F(x),0)$ and
the corresponding $c\in(0,1]$ such that $F(x)=cp+(1-c)w$. From
hypotheses (A) and (B) we have
$$
0=H(x,F(x))\leq cH(x,p)+(1-c)H(x,w)\,,
$$
and from \eqref{doccia1} we deduce
$$
H(x,p)\geq\frac{c-1}{c}H(x,w)\geq 0\,.
$$
The remaining cases can be treated similarly.
%We obtained that for any $x\in \mathbb T$, $\varphi$ is a viscosity
%solution of the Hamilton-Jacobi equation \eqref{HJgen} at $x$. This
%means that $\varphi$ is a viscosity solution of \eqref{HJgen}. $\
%\square$
\end{proof}
\medskip

We can finally conclude:

\medskip

\noindent {\sl Proof of  Theorem \ref{genteo}}. We only need to show
that the function $\Phi$ satisfies conditions $(a),(b),(c)$ and
$(d)$ of Lemma \ref{grufalo}.

The validity of conditions $(a),(b)$ and $(c)$ follows directly from
Theorem \ref{ruspa} with the identification $O=U$. Let us show that
also condition $(d)$ is satisfied.  To this aim, we consider
 a maximal connected component  $(u_i^-,u^+_i)$ of $U$. Then by
definition we have
$\Phi(u_i^-)=\Phi(u^+_i)=\min\left\{0,-S(1)\right\}$ and
$\Phi(u)<\min\left\{0,-S(1)\right\}$ for any $u\in (u_i^-,u^+_i)$.
In particular, for any $u\in (u_i^-,u^+_i)$ we can write
\begin{align}
& 0> \Phi(u)-\min\left\{0,-S(1)\right\}= \Phi (u)- \Phi(u_i^-) = \int_{u^-_i}^uF(z)dz\,,\\
&0> \Phi(u)-\min\left\{0,-S(1)\right\}=  \Phi (u)- \Phi(u_i^+)  =
\int_{u^+_i}^uF(z)dz\,.
\end{align}
Due to the continuity of $F$, the above expressions imply that
$F(u_i^-) \leq 0$ and $F(u_i^+) \geq 0$. \qed

\bigskip

\bigskip

 \noindent
{\bf  Acknowledgements}. The authors kindly thank L. Bertini, A. De
Sole,  G. Jona--Lasinio and E. Scoppola for useful discussions. A.F.
acknowledges  the financial support of  the European Research
Council through the ``Advanced Grant''  PTRELSS 228032. D.G.
acknowledges the financial support of PRIN 20078XYHYS$\underline{\
}$003 and thanks the Department of Physics of the University ``La
Sapienza" for the kind hospitality.

\end{document}